\documentclass[12pt]{article}
\usepackage{amsmath, amsfonts, amssymb, amsthm}
\usepackage[section]{algorithm}
\usepackage{algorithmic}
\usepackage{authblk}
\usepackage{caption}
\usepackage[T1]{fontenc}
\usepackage{graphicx}
\usepackage{epstopdf}
\usepackage{color}
\usepackage{geometry}
\usepackage{bbm}
\theoremstyle{plain}

\usepackage{array}
\newcolumntype{?}{!{\vrule width 2pt}}

\makeatletter

\newcommand{\0}{\mathbf{0}}
\newcommand{\Id}{\mathbb I}
\def\det   {\mathop{\rm det}\nolimits}

\def\diag {\mathop{\rm diag}\nolimits}
\def\rank {\mathop{\rm rank}\nolimits}

\newcommand{\ii}{\mathfrak{i}}

\usepackage{color}
\usepackage{geometry}
\usepackage{subcaption}
\usepackage{amsthm}

\usepackage{array}

\usepackage{tikz}
\usepackage{geometry}

\tikzset{
	treenode/.style = {shape=rectangle, rounded corners,
		draw, align=center,
		top color=white, bottom color=blue!20},
	root/.style     = {treenode, font=\Large, bottom color=red!30},
	env/.style      = {treenode, font=\ttfamily\normalsize},
	dummy/.style    = {circle,draw}
}

\def\roff  {\mbox{\boldmath$\varepsilon$}}

\newtheorem{theorem}{Theorem}[section]

\newtheorem{example}{Example}[section]
\newtheorem{remark}{Remark}[section]
\newtheorem{proposition}{Proposition}[section]

\title{	New numerical algorithm for deflation of infinite and zero eigenvalues and full solution of quadratic eigenvalue problems}

\author[1]{Zlatko Drma\v{c}}
\author[2]{Ivana \v{S}ain Glibi\'{c}}
\affil[1]{Department of Mathematics, Faculty of Science, University of Zagreb}
\affil[2]{Department of Mathematics, Faculty of Science, University of Zagreb}

\date{}

\numberwithin{equation}{section}

\begin{document}
	\maketitle
\begin{abstract}
	This paper presents a new method for computing all eigenvalues and  eigenvectors of  quadratic matrix pencil $Q(\lambda)=\lambda^2 M + \lambda C + K$. It is an upgrade of the \texttt{quadeig} algorithm by Hammarling, Munro and Tisseur  (ACM Transactions on Mathematical Software 39 (3), 2013.), which attempts to reveal and remove by deflation certain number of zero and infinite eigenvalues before QZ iterations. Proposed modifications of the \texttt{quadeig} framework are designed to enhance backward stability and to make the process of  deflating infinite and zero eigenvalues more numerically robust. In particular, careful preprocessing allows scaling invariant/component-wise backward error and thus better condition number. Further, using an upper triangular version of the Kronecker canonical form enables deflating additional infinite eigenvalues, in addition to those inferred from the rank of $M$.   Theoretical analysis and empirical evidence from thorough testing of the software implementation confirm superior numerical performances of the proposed method.
\end{abstract}

\maketitle	

\section{Introduction}

We propose a new numerical method for computing all eigenvalues with the corresponding left and right eigenvectors of the $n\times n$ quadratic eigenvalue problem
\begin{equation}\label{qep}
Q(\lambda)x = (\lambda^2 M + \lambda C + K)x = \0,\;\; M, C, K \in\mathbb{C}^{n\times n}.
\end{equation}
Such problem is usually associated with a damped system $M \ddot q(t) + C \dot q(t) + K q(t)=f(t)$, which arises in many applications in science and engineering. For an excellent review of applications and numerical methods we refer to \cite{Tisseur-Meerbergen-SIREW-2001}.

The dimension $n$ can be large, e.g. in millions \cite{Mehrmann2011}, and the matrices usually possess some structure such as e.g. sparsity that depends on the geometry of the underlying physical model, discretization method etc. In such cases, one is often interested in finding only those eigenvalues (with the corresponding eigenvectors) that meet some particular specifications; examples include the tasks of finding $k\ll n$ eigenvalues closest to some target value $\xi\in\mathbb{C}$, or all eigenvalues belonging to some prescribed $\Omega\subset\mathbb{C}$. In some other applications, the dimension $n$ may be moderate with non-sparse coefficient matrices and all eigenvalues, optionally with the corresponding eigenvectors, are requested.

Even in the large scale problems with sparse (or otherwise structured) coefficient matrices, an iterative method, such as e.g. the Second Order Arnoldi (SOAR \cite{bai2005soar}) uses the  Rayleigh-Ritz projection and the full solution of the projected problem $(\lambda^2 [X^* MX] + \lambda [X^* CX] + [X^* KX])v=\0$ of moderate size is needed. Here $X$ is an $n\times k$ orthonormal matrix that should span a subspace rich in the desired spectral information, and the dimension $k$ is moderate, $k\ll n$, so that the full solution with dense $k\times k$ coefficient matrices  is feasible. 

We focus to the complete solution of (\ref{qep}), i.e. finding all eigenvalues and eigenvectors; the dimension $n$ is assumed moderate. This is usually solved by linearization -- transformation to an equivalent eigenvalue problem for linear pencils, such as e.g.
\begin{equation}\label{linearization-C2}
C_2(\lambda) = \begin{pmatrix}
C & -\Id_n \\
K & \0_n 
\end{pmatrix} - \lambda \begin{pmatrix}
-M & \0_n\\
\0 & -\Id_n
\end{pmatrix} \equiv A - \lambda B.
\end{equation}
From the space of linearizations, one is advised to use the one that preserves, as much as possible, structural (spectral) properties of the original problem. This is important because numerical methods  can leverage particular structure when applied to an appropriate linearization. It should be stressed that linearizations algebraically equivalent to (\ref{qep}) are not numerically equivalent to the quadratic problem in the sense of sensitivity when solved in finite precision arithmetic.

For the linearized problem (\ref{linearization-C2}), the method of choice  is  the QZ method. In some cases this may run into difficulties, in  particular if $M$ is exactly or nearly rank deficient, which leads to (numerically) infinite eigenvalues. Whether the QZ is troubled by the presence of the infinite eigenvalues is a tricky question. On one hand, from an expert point of view, \cite{Watkins:2000:QZ-infinite} shows that the presence of infinite eigenvalues cannot derange the transmission of shifts during the QZ iterations. On the other hand, the filigran-work with plane rotations \cite{Watkins:2000:QZ-infinite} that ensures this robustness is often replaced (in software implementations) with faster block-oriented procedures, and infinite eigenvalues may remain latent and come as spurious absolutely large values, with negative impact on the accuracy of the computed genuine finite eigenvalues. In the trade-off between accuracy and speed, the developers usually choose speed and \emph{flops} seem more important than the number of correct digits in the output. However, the block oriented computation does provide considerably better performances, and this is an important aspect of software development. 

So, in general, it is advantageous to deflate the infinite eigenvalues early in the computational scheme, using computational tools that can be efficiently implemented on the state of the art computers. Similarly, if $K$ is rank deficient, then its null space provides eigenvectors for the eigenvalue $\lambda=0$, and removing it in a preprocessing phase facilitates more efficient computation of the remaining eigenvalues.  In both cases a nontrivial decision about the numerical rank has to be made. 

This motivated \cite{Hammarling:QUADEIG} to develop a new deflation scheme that removes  $n-r_M$ infinite and $n-r_K$ zero eigenvalues, where $r_M = \rank{M}$ and $r_K = \rank{K}$. The remaining generalized linear eigenvalue problem is of the dimension $r_M+r_K$; it may still have some infinite and zero eigenvalues, and their detection then depends on the performance of the QZ algorithm. The computation is done in the framework of the linearization (\ref{linearization-C2}), and the proposed approach has resulted in the \texttt{quadeig} algorithm. 


For a pair $(\lambda,x)$ of an eigenvalue and its right eigenvector, the optimal backward error to test whether $(\lambda,x)$ is acceptable is defined as (see \cite[\S 2.2]{tisseur2000backward})
\begin{equation}\label{eq:back-err-0}
\eta(\lambda\neq\infty, x)=\frac{\| (\lambda^2 M + \lambda C + K )x\|_2}{(|\lambda|^2 \|M\|_2+|\lambda|\|C\|_2 + \|K\|_2)\|x\|_2};\;\; \eta(\lambda=\infty,x) = \frac{\|Mx\|_2}{\|M\|_2\|x\|_2},
\end{equation}
where, by definition, $\lambda=\infty$ is an eigenvalue of (\ref{qep}) if and only if $\mu=0$ is an eigenvalue of the reversed pencil $\mu^2 K + \mu C + M$.
The backward stability (or the residual) are usually measured by $\eta(\lambda,x)$ and the computation is deemed stable/acceptable if $\eta(\lambda,x)$ is, up to a moderate factor, of the order of the machine precision. This is justified by the fact that we have an exact eigenpair of a nearby pencil with coefficient matrices $\widetilde{M}=M+\delta M$, $\widetilde{C}=C+\delta C$, $\widetilde{K}=K+\delta K$, where
\begin{equation*}
\eta(\lambda,x)\! =\! \min\{ \epsilon  :  (\lambda^2 \widetilde{M}+\lambda \widetilde{C}+\widetilde{K})x \! =\!\0,\;\|\delta M\|_2\leq\! \epsilon \|M\|_2, 
\|\delta C\|_2\leq\! \epsilon \|C\|_2, 
\|\delta K\|_2\leq\! \epsilon \|K\|_2
\}.
\end{equation*}

Using this, essentially only necessary, condition is often used as sufficient for declaring the computed result reliable. The theoretical advances to identify condition numbers as measures of sensitivity of the eigenvalues \cite[\S 2.3]{tisseur2000backward} are rarely used as part of numerical software solution. Unfortunately, such an approach is potentially prone to errors. The best way to illustrate the difficulty is by an example.

\begin{example}\label{EX-Intro-NLEVP-MM}
{
Consider the model of  a two-dimensional three-link mobile manipulator (robot for e.g. facade cleansing, construction) \cite{HOU19961554}.\footnote{This  example is included in the NLEVP benchmark collection \cite{betcke2010nlevp}.} Linearisation around a particularly chosen working point yields a $5 \times 5$ quadratic matrix polynomial (\ref{qep}) with the coefficient matrices
\begin{equation}\label{eq:mob-man-MCK}
M = \begin{pmatrix}
M_0 & \0_{3\times 2}\\
\0_{2\times 3} & \0_{2\times 2}
\end{pmatrix}, \hspace{2 mm} C = \begin{pmatrix}
C_0 & \0_{3\times 2} \\
\0_{2\times 3} & \0_{2\times 2}
\end{pmatrix}, \hspace{2 mm} K = \begin{pmatrix}
K_0 & -F^T_0\\
F_0 & \0_{2\times 2}
\end{pmatrix},
\end{equation}
where $F_0 = \left(\begin{smallmatrix}
1 & 0 & 0\\
0 & 0 & 1
\end{smallmatrix}\right)$ enforces certain constraints, $M_0 = \left(\begin{smallmatrix}
18.7532 & -7.94493 & 7.94494\\
-7.94493& 31.8182& -26.8182\\
7.94494 &-26.8182 & 26.8182\\
\end{smallmatrix}\right)$, 
$$
C_0 = \left(\begin{smallmatrix}
-1.52143& -1.55168 & 1.55168\\
3.22064 & 3.28467 & -3.28467\\
-3.22064& -3.28467& 3.28467
\end{smallmatrix}\right) ,\;\;
K_0 = \left(\begin{smallmatrix}
67.4894& 69.2393 & -69.2393\\
69.8124 & 1.68624& -1.68617\\
-69.8123 & -1.68617 & -68.2707
\end{smallmatrix}\right) .
$$
The matrix $K$ has full rank, and the rank of $M$ is $r_M=3$. This means that there are at least $n-r_M = 2$ infinite eigenvalues. We computed the eigenvalues using the \texttt{quadeig}  algorithm  \cite{Hammarling:QUADEIG} with default options as follows:\footnote{Al computations were done using Matlab 8.5.0.197613 (R2015a) under Windows 10.  The roundoff unit in Matlab is $\texttt{eps}\approx 2.2\cdot^{-16}$.}
\begin{equation}\label{eq:quadeig:mob-man}
\begin{array}{ll}
\lambda_1 =  \texttt{-5.1616e-002 -2.2435e-001i} & \lambda_6 = \texttt{1.9573e+06 - 3.3993e+06i} \cr 
\lambda_2 =  \texttt{-5.1616e-002 +2.2435e-001i} & \lambda_7 = \texttt{1.9573e+06 + 3.3993e+06i} \cr 
\lambda_3 =  \texttt{2.7700e+05 - 4.7988e+05i} & \lambda_8 = \texttt{-3.9305e+06 + 0.0000e+00i} \cr 
\lambda_4 =  \texttt{2.7700e+05 + 4.7988e+05i}  & \lambda_9 = \texttt{Inf}  \cr
\lambda_5 =  \texttt{-5.5417e+05 + 0.0000e+00i } & \lambda_{10} = \texttt{Inf}
\end{array}
\end{equation}
%
%
%
\noindent Next, we compute the eigenvalues by a direct application of the QZ algorithm to (\ref{linearization-C2}). The QZ algorithm found $8$ infinite and two finite eigenvalues:
\begin{equation}\label{eq:polyeig:mob-man}
	\begin{array}{l}
	\lambda_1 = \texttt{-5.161621336216381e-002 -2.243476109085836e-001i} \cr 
	\lambda_2 = \texttt{-5.161621336216381e-002 +2.243476109085836e-001i}\cr
	\lambda_3 = \cdots = \lambda_8=\texttt{Inf}, \;\;\;\lambda_9=\lambda_{10} = \texttt{-Inf}.
	\end{array}
\end{equation}
If we check the backward errors, then both (\ref{eq:quadeig:mob-man}) and (\ref{eq:polyeig:mob-man}) (with the corresponding right eigenvectors) pass the backward stability test with $\eta(\lambda,x)$ at the level of machine roundoff for all eigenvalues.
Yet, (\ref{eq:quadeig:mob-man}) has six finite eigenvalues more than  (\ref{eq:polyeig:mob-man}), and those six values are $O(10^6)$ in modulus.
How can we tell whether those values correspond to genuine $O(10^6)$ big finite eigenvalues, or whether they should be declared spurious? After all, the backward error (\ref{eq:back-err-0}) below $n\cdot \texttt{eps}$ seems reassuring, but the same holds for (\ref{eq:polyeig:mob-man}) as well.  In a particular application, engineer's intuition may help, but for a black-boxed numerical software this remains an important and difficult problem. This clearly shows that the problem is nontrivial; one can easily imagine similar ambiguity when solving a problem with, say,  $n > 10^4$, where different methods give completely different results, and all with small backward error (\ref{eq:back-err-0}). 
\hfill $\diamond$ 
}
\end{example}
%

An important message of Example \ref{EX-Intro-NLEVP-MM} is that backward stability measured by (\ref{eq:back-err-0}) is only a necessary/desirable condition which may not be enough for reliable computation.
 More strengthen measure for backward stability is required, such as the component-wise backward error (residual)
 \begin{eqnarray}
\!\!\!\! \omega(\lambda,x) &=& \min\{ \epsilon  :  (\lambda^2 \widetilde{M}+\lambda \widetilde{C}+ \widetilde{K})x=\0,\;\;|\delta M|\leq \epsilon |M|, 
 |\delta C|\leq \epsilon |C|, 
 |\delta K|\leq \epsilon |K|
 \} \nonumber \\  
 &=& \max_{i=1:n} \frac{|(\lambda^2 M + \lambda C + K)x|_i}{((|\lambda|^2|M| + |\lambda||C| + |K|)|x|)_i} , \;\widetilde{M}\!=\! M\! +\!\delta M, \widetilde{C}\! =\! C\! +\! \delta C, \widetilde{K}\! =\! K\! +\! \delta K, \label{eq:ComponentwiseBE}
 \end{eqnarray}
 and the sensitivity of the eigenvalues must be estimated using appropriate condition number, see \cite{higham1998structured},  \cite{higham2007backward}, \cite{Higham-Mackey-Scaling-2008}.
 Using $\omega(\lambda, x)$ in Example \ref{EX-Intro-NLEVP-MM} provides valuable information. Although \texttt{quadeig} is considered better that \texttt{polyeig} \cite{Hammarling:QUADEIG}, in this example the QZ algorithm (including its preprocessing) seems more successful, see Figure \ref{fig:EX-mob-manip-2}.

\begin{figure}[H]
	\centering
	\includegraphics[width=0.75\textwidth, height=1.6in]{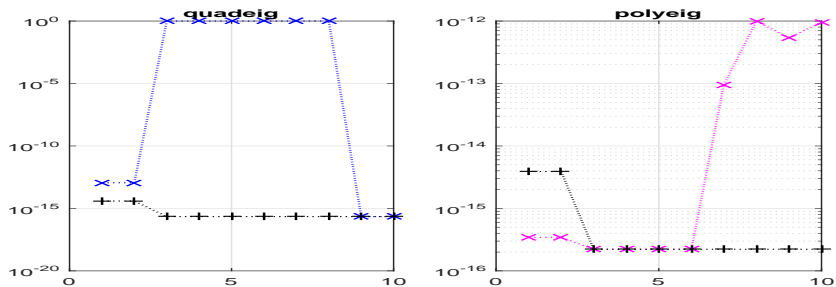}
	\caption{\label{fig:EX-mob-manip-2} (Example \ref{EX-Intro-NLEVP-MM}.) The backward errors (\ref{eq:ComponentwiseBE}). (We display in log scale $\max(\omega(\lambda,x),\texttt{eps})$.) \emph{Left panel}: Eigenvalues (\ref{eq:quadeig:mob-man})(\textcolor{blue}{x}) computed by \texttt{quadeig}. \emph{Right panel}: Eigenvalues (\ref{eq:polyeig:mob-man})(\textcolor{magenta}{x}) computed by \texttt{polyeig}. In both graphs, the pluses denote the backward error for the method \texttt{KVADeig} proposed in this paper.}	
\end{figure}	 
 
Entry--wise backward stability (in the sense that the relative backward error in each matrix entry of $M$, $C$, $K$ is small) is in general not
feasible; not even in the case $n=1$, see
\cite{Mastronardi-VanDooren-quad-back-err-ETNA}. For practical usage, a mixed analysis (backward-forward) might be more appropriate, with more general reference values for measuring size of the entry-wise/column-wise backward errors, e.g. $|\delta M|\leq \epsilon \Xi_M$, where $\Xi_M \geq \0$.


To the best of our knowledge, \texttt{quadeig} \cite{Hammarling:QUADEIG} is the only published algorithm for the full solution of (\ref{qep}) that is supported by a backward error analysis and thoroughly tested, with publicly available software implementation.
Our main contribution in this work is an upgrade of of the framework set by \texttt{quadeig}, and it has two main components. First, we treat the numerical rank issue and te backward error more carefully, so that our proposed method has better structured backward error in the sense of \cite{higham1998structured}. This is confirmed by theoretical analysis and numerical experiments.  Second, we interpret the deflation strategy deployed in \texttt{quadeig} as the beginning step of computing the Kronecker canonical structure associated with the zero/infinite eigenvalue, which motivates taking further steps toward the Kronecker canonical form (KCF), thus opening possibility for removing additional infinite eigenvalues. Indeed, we show that our proposed method is capable of discovering infinite eigenvalues that are beyond the reach of \texttt{quadeig}. A preview of the performance of the  new proposed method, designated as \texttt{KVADeig}\footnote{Pronounced as \texttt{quadeig}.}, is given in Figure \ref{fig:EX-mob-manip-2}.

\subsection{Outline of the paper}
The rest of the paper is organized as follows. In \S \ref{S=RRD} we review, for the readers convenience, numerical computation of rank revealing decompositions. In particular, we review the pivoted QR factorization and the complete orthogonal (URV) decomposition, and in \S \ref{ss:COD} we emphasize the importance of pivoting for more structured backward error bounds in case of unbalanced rows in data matrices.  Further, in \S \ref{s:kronecker}, we recall the Van Dooren's algorithm \cite{van1979computation} for computing the KCF of general matrix pencil $A-\lambda B$. We rewrite the Van Dooren's algorithm in an upper triangular form, and use the rank revealing QR factorization instead of the SVD; this procedure is at the core of our new algorithm. 
Numerical equivalence of linearization and the original quadratic problem in the sense of backward stability is a delicate issue and special preprocessing by parameter scaling is recommended; see an enlightening discussion and examples in
\cite{Higham-Mackey-Scaling-2008}, and implementation details in \texttt{quadeig} \cite{Hammarling:QUADEIG}. Scaling is briefly reviewed in \S \ref{S=S+B}; in \S \ref{SS=B} we also propose balancing matrix entries by diagonal scalings of the coefficients of (\ref{qep}).
In \S \ref{S=New-Deflation} we describe our modified approach to the \texttt{quadeig} type preprocessing. The goal is to adapt the rank revealing steps to the form of the KCF described in \S \ref{SS=KCF-QR}. In \S \ref{SS=KVADeig-backward} we provide detailed backward error analysis of the first two steps of the reduction to the KCF, and prove backward stability in terms of the original data matrices.
The proof is instructive because it shows how the modifications introduced in \texttt{KVADeig} contribute to backward stability. The global structure of the new algorithm is summarized in \S \ref{S=KVADEIG-global}. Numerical experiments are presented in \S \ref{S=NUMEX}. Concluding remarks and comments on ongoing and future work are given in  \S \ref{S=Conclusion}.

\section{Rank revealing decompositions}\label{S=RRD}
Since detecting zero or infinite eigenvalues is based on numerical rank decision, we briefly discuss rank revealing decompositions (RRD, see \cite{dgesvd-99}). For a general $m\times n$ matrix $A$, we say that $A = X D Y^*$ is a rank revealing decomposition if both $X$ and $Y$ are of full column rank and well conditioned, and $D$ is diagonal nonsingular. In practical computation, such a decomposition is computed only approximately and we have $A+\delta A = \widetilde{X} \widetilde{D}\widetilde{Y}^*$, where $\delta A$ denotes initial uncertainty and/or the backward error that corresponds to the numerically computed $\widetilde{X}$, $\widetilde{D}$ and $\widetilde{Y}$. Hence, any decision on the rank actually applies to $A+\delta A$. We refer to \cite{Golub-Klema-Stewart-Numerical_rank} for a more in depth discussion and definition of a numerical rank. 

Since the full rank matrices are open dense set in $\mathbb{C}^{m\times n}$ ($\mathbb{R}^{m\times n}$), it is unlikely that, in general, the rank deficiency  will be determined correctly in finite precision computation. Furthermore, in many applications the matrix has been already contaminated by errors (initial uncertainty, previous computational steps etc.) and a firm statement about its rank is illusory. Since the zero and the infinite eigenvalues are rooted in rank deficiency, it is clear that revealing them numerically is a delicate issue. 


In this section, we review the pivoted QR factorization, and the complete orthogonal decomposition (URV, ULV), which will be used as rank revealing decompositions. We discuss the backward stability, which will be an important ingredient in the analysis of the proposed algorithm.

\subsection{QR factorization with column and complete pivoting}\label{SS=QRF-review}
QR factorization with column pivoting is a tool of trade in a plethora of applications, in particular when the numerical rank of a matrix plays an important role. It is the main computational tool in the reduction procedure in the \texttt{quadeig} algorithm as well.

 Particularly successful is the Businger--Golub 
\cite{bus-gol-65} pivot strategy which, for
$A\in \mathbb{C}^{m\times n}$, computes
a permutation matrix $P$, a unitary $Q$ and an $\min(m,n)$ upper triangular (trapezoidal if $m<n$) matrix
$R$ such that
\begin{equation}\label{QRCP}
A P = Q \begin{pmatrix} R \cr \0 \end{pmatrix},\;\;\mbox{where}\;\;|R_{ii}|^2\geq {\sum_{k=i}^j
	|R_{kj}|^2},\;\;\mbox{for all}\;\;1\leq i\leq j \leq n.
\end{equation}
Here, for the sake of brevity, we consider only the case $m\geq n$. If $m<n$, then $R$ is $m\times n$ upper trapezoidal and the zero block in (\ref{QRCP}) is void.
If $\rho=\mathrm{rank}(A)$, then $\prod_{i=1}^\rho R_{ii}\neq 0$ and $R(\rho+1:n,\rho+1:n)=\0$. If $k\in\{1,\ldots, n\}$, and if we introduce the block partition
\begin{equation}\label{eq:R:block-part}
R = \left(\begin{smallmatrix} R_{[11]} & R_{[12]} \cr \0 & R_{[22]}\end{smallmatrix}\right),\;\;R_{[11]}\in\mathbb{C}^{k\times k}, 
\end{equation}
\begin{equation}\label{eq:A_k:QRCP}
\mbox{then}\;\;\;(A + \Delta A)P \equiv (A - Q \left(\begin{smallmatrix} \0 & \0 \cr \0 & R_{[22]} \cr \0 & \0\end{smallmatrix}\right)\! P^T ) P = Q\! \left(\begin{smallmatrix} R_{[11]} & R_{[12]} \cr \0 & \0 \cr \0& \0\end{smallmatrix}\right)\!,\;\Delta A \equiv - Q \!\left(\begin{smallmatrix} \0 & \0 \cr \0 & R_{[22]} \cr \0 & \0\end{smallmatrix}\right) \! P^T .
\end{equation}
If $k$ is such that $R_{[11]}$ is of full rank, then $A+\Delta A$ is of rank $k$ and 
\begin{equation}\label{eq:DA-R22}
\|\Delta A\|_F = \|R_{[22]}\|_F\leq \sqrt{n-k} |R_{k+1,k+1}|.
\end{equation}
Hence, if $\gamma >0$ is a given threshold, and if we can find an index $k$ ($1\leq k<n$) such that 
\begin{equation}\label{eq:QRCP:gamma-k}
\sqrt{n-k} |R_{k+1,k+1}|/\|A\|_F \leq \gamma,
\end{equation}
then $A$ is $\gamma$-close to the rank $k$ matrix $A+\Delta A$,  whose pivoted QR factorization (\ref{eq:A_k:QRCP}) is obtained from (\ref{QRCP}) by setting in the partition (\ref{eq:R:block-part}) the block  $R_{[22]}$ to zero. Of course, we would take the smallest possible $k$ that satisfies (\ref{eq:QRCP:gamma-k}). 
The essence of rank revealing capability of the factorization is in the fact that such a $k$ will very likely be visible on the diagonal of $R$ if $A$ is close to a rank $k$ matrix. This is due to the fact that the $|R_{ii}|$'s mimic the distribution of the singular values of $A$, \cite{ste-97}, \cite{ste-97-qlp}.


\begin{remark}
{
The rank--$k$ approximation $A+\Delta A$ defined in (\ref{eq:A_k:QRCP}) in general does not share the optimality property 	of the closest rank $k$ matrix $A_k$ from the SVD, but it has one distinctive feature: it always matches $A$ exactly at the selected $k$ columns, while $A_k$ in general does not match any part of $A$. 
}
\end{remark}	
\begin{remark}
	{
An efficient and numerically reliable implementation of (\ref{QRCP}) is available e.g. in LAPACK \cite{LAPACK} in the function \texttt{xGEQP3}, which is also under the hood of the Matlab's function \texttt{qr()}; for the numerical and software details we refer to \cite{drmac-bujanovic-2008}.
}
\end{remark}
\subsubsection{Backward stability}
Computation of the QR factorization in finite precision arithmetic is backward stable: for the computed factors $\widetilde{P}$, $\widetilde{Q}$, $\widetilde{R}$, there exist a backward error $\delta A$ and a unitary matrix $\widehat{Q}$ such that 
\begin{equation}\label{QRCP-backward-error}
(A+\delta A) \widetilde{P} = \widehat{Q}\begin{pmatrix} \widetilde{R} \cr \0 \end{pmatrix},\;\; \|\delta A\|_F \leq \epsilon_{qr} \|A\|_F,\;\; \| \widetilde{Q}-\widehat{Q}\|_F \leq \epsilon_{qr} .
\end{equation}
Here $\epsilon_{qr}$ is bounded by a moderate function of the matrix dimensions times the machine roundoff $\roff$.
In fact, the backward stability is even stronger -- the backward error in each column is small relative to its norm, 
\begin{equation}\label{QRCP-backward-error-columns}
\|\delta A(:,i)\|_2 \leq \epsilon_{qr} \|A(:,i)\|_2,\;\;i=1,\ldots, n.
\end{equation}
This is an important feature if some columns of $A$ are, by its nature, much smaller than the largest ones (different weighting factors, different physical units); (\ref{QRCP-backward-error-columns}) assures that the computed factorization contains the information carried by small columns of $A$. While (\ref{QRCP-backward-error}, \ref{QRCP-backward-error-columns})  hold independent of pivoting,  pivoting is important for the accuracy of the computed factorization, and for the rank revealing. 

If, for a suitable partition of $\widetilde{R}$, analogous to (\ref{eq:R:block-part}), we can determine $k$ such that $\widetilde{R}_{[22]}$ can be chopped off, we have 
 \begin{eqnarray}
 && (A+\delta A + \Delta A) \widetilde{P} = \widehat{Q}\left(\begin{smallmatrix} \widetilde{R}_{[11]} & \widetilde{R}_{[12]} \cr \0 & \0 \cr \0& \0\end{smallmatrix}\right),\;\; \|\delta A\|_F \leq \epsilon_{qr} \|A\|_F,\;\; \| \widetilde{Q}-\widehat{Q}\|_F \leq \epsilon_{qr}, \label{QRCP-backward-error-chopped}\\
 && \Delta A \equiv - \widehat{Q} \!\left(\begin{smallmatrix} \0 & \0 \cr \0 & \widetilde{R}_{[22]} \cr \0 & \0\end{smallmatrix}\right) \! \widetilde{P}^T  \label{eq:DA-tilde} \;\;\mbox{(estimated using (\ref{eq:DA-R22})),}\;\;
 (\Delta A \widetilde{P})(:,1:k)=\0.
 \end{eqnarray}
\subsubsection{Complete pivoting}\label{SSS=QRCP}
In some applications (e.g. weighted least squares) the rows of the data matrix may vary over several orders of magnitude, and it is desirable to have backward error that can be bounded row-wise analogously to (\ref{QRCP-backward-error-columns}). A pioneering work is done by Powell and Reid \cite{pow-rei-68}, who introduced QR factorization with complete pivoting. More precisely, in a $j$-th step, before deploying the Householder reflector to annihilate below-diagonal entries in the $j$-th column, row swapping is used to bring the absolutely largest entry to the diagonal position. As any pivoting, this precludes efficient blocking and using BLAS 3 level primitives. 

\cite{bjorck-NMLS-book-1996} noted that the complete pivoting can be replaced with an initial sorting of the rows of $A$ in monotonically decreasing order with respect to the $\ell_\infty$ norm. If $P_r$ is the corresponding row permutation matrix, and if we set $A:= P_r A$, then
\begin{equation}\label{sort}
\| A(1,:) \|_{\infty} \geq \| A(2,:) \|_{\infty} \geq \ldots \geq \| A(m,:) \|_{\infty},
\end{equation}
and we proceed with the column pivoted factorization (\ref{QRCP}). An advantage of replacing the dynamic complete pivoting of Powell and Reid with the initial pre-sorting (\ref{sort}) followed by column pivoted QRF (\ref{QRCP}) is more efficient software implementation. 
 An error analysis of this scheme and Householder reflector based QR factorization is given by Cox and Higham \cite{cox-hig-98}. 
\begin{equation}\label{eq:dA-rows-1}
\max_{i=1:m} \frac{\|\delta A(i,:)\|_\infty}{\|A(i,:)\|_\infty} \leq \epsilon_{qr} \max_{i=1:m} \frac{\alpha_i}{\|A(i,:)\|_\infty}\equiv \epsilon_{qr}^{\rightrightarrows},\;\;\mbox{where}\;\;
\alpha_i = \max_{j,k} |\widetilde{A}^{(k)}_{ij}| ,
\end{equation}
and $\widetilde{A}^{(k)}$ is the $k$th computed (in finite precision arithmetic) intermediate matrix in the Householder QR factorization. As a result of initial row ordering and the column pivoting, \cite{cox-hig-98} shows that 
\begin{equation}\label{eq:dA-rows-2}
\alpha_i \leq \left\{ \begin{array}{ll} \sqrt{m-i+1}(1+\sqrt{2})^{i-1} \|A(i,:)\|_\infty, & i\leq n  \cr (1+\sqrt{2})^{n-1}\|A(i,:)\|_\infty, & i > n\end{array}\right. ,
\end{equation}
where the factor $(1+\sqrt{2})^{n-1}$ is almost never experienced in practice.

\begin{remark}\label{REM:QRCPP}
{
If we write the completely pivoted factorization as 
$P_r A P_c = Q \left(\begin{smallmatrix} R \cr \0 \end{smallmatrix}\right)$, then $A P_c = (P_r^T Q) \left(\begin{smallmatrix} R \cr \0 \end{smallmatrix}\right)
$
is the column pivoted QR factorization (since $P_r$ is orthogonal); hence we can ease the notation of the complete pivoting and absorb the row permutation in the orthogonal factor, $Q\leftarrow P_r^T Q$. Note that the row pivoting brings nothing new to the rank revealing property that is encoded in the triangular factor (this is because of the essential uniqueness of the factorization). However, row pivoting has the important role in the numerical accuracy of the computed factorization.
}
\end{remark}

 \subsection{The complete orthogonal factorization (ULV, URV)}\label{ss:COD}
 Suppose that in the QR factorization (\ref{QRCP}), the matrix $A$ is of rank $k < n$, so that in the block partition (\ref{eq:R:block-part}) $R_{[22]}=\0$. In many instances, it is convenient to compress the trapezoidal matrix $(R_{[11]}\; R_{[12]})$ to (upper or lower) triangular form by an additional RQ or LQ factorization.
The resulting complete orthogonal decomposition, often used in least squares computation \cite{law-han-74}, \cite{Hough-Vavasis-CODWLS-1997}, is one of the computational subroutines in the \texttt{quadeig} algorithm.
 
 The LQ factorization of $(R_{[11]}\; R_{[12]})$ is equivalent to computing the QR factorization 
 \begin{equation}\label{eq:LQ-ULV}
 \left(\begin{smallmatrix}
 R^*_{[11]}\\
 R^*_{[12]}
 \end{smallmatrix}\right) = Z_R^* \left(\begin{smallmatrix}
 T^*_{[11]}\\
 \0
 \end{smallmatrix}\right),
 \end{equation}
 where $T_{11}\in \mathbb{C}^{k\times k}$ is lower triangular and nonsingular. By a composition of these two steps we get the so called complete orthogonal decomposition of $A$
 \begin{equation}\label{eq:A=QTZ}
 A = Q \left(\begin{smallmatrix}
 T_{[11]} & \0\\
 \0 & \0
 \end{smallmatrix}\right)Z,\;\;\mbox{where $Z = Z_RP^T$.}
 \end{equation}
  In the QR factorization, the matrix $A$ is multiplied from the left by a sequence of unitary transformations. Hence, there is no mixing of the columns; we can analyse the process for each column separately; that is why the column-wise backward error bound (\ref{QRCP-backward-error-columns}) is natural and straightforward to derive.
  The transformations from the right in the pivoted QR factorization are the error free column interchanges.
  
  On the other hand, (\ref{eq:A=QTZ}) involves nontrivial two--sided transformations of $A$, and more careful implementation and error analysis are needed to obtain backward stability similar to the one described in \S \ref{SS=QRF-review}. More precisely, the backward error in $A$ of the type (\ref{QRCP-backward-error}) is a priori assured, because the transformations are unitary/orthogonal, but finer structure such as (\ref{QRCP-backward-error-columns}) requires some more work. A sharper backward error bound is attractive because the complete orthogonal decomposition is used in the reduction procedure in \texttt{quadeig}, and in our new proposed method. 
  
  Based on the discussion in \S \ref{SSS=QRCP}, we compute with complete pivoting both the rank revealing QR factorization of $A$ and (\ref{eq:LQ-ULV}).
 The resulting scheme for computing the complete orthogonal decomposition is summarized in Algorithm \ref{a:COD}.
 \begin{algorithm}[h!]
 	\caption{Complete orthogonal decomposition of $A\in\mathbb{C}^{m\times n}$}
 	\label{a:COD}
 	\begin{flushleft}
 		\textbf{INPUT:} $A\in\mathbb{C}^{m\times n}$ \\
 		\textbf{OUTPUT:} $Q$, $T_{[11]}$, $Z$, so that $A = Q \left(\begin{smallmatrix}
 		T_{[11]} & \0\\
 		\0 & \0
 		\end{smallmatrix}\right) Z^*$
 	\end{flushleft}
 	\begin{algorithmic}[1]
 		\STATE Compute permutation $P_2$ such that $\|e_i^TP_2A\|_{\infty} \geq \|e^T_{i+1}P_2A\|_{\infty}$, $i=1,\ldots,m-1$.
 		\STATE Compute the column pivoted QR factorization $(P_2^T A)P = Q \left(\begin{smallmatrix}
 		R\\
 		\0
 		\end{smallmatrix}\right)$, $R\in\mathbb{C}^{\rho\times n}$, $\rho=\mathrm{rank}(A)$.
 		\STATE Compute the QR factorization with complete pivoting:
 		$(\Pi_2 R^*)\Pi_1 
 		= Z_R \begin{pmatrix}
 		T^*_{[11]}\\
 		\0
 		\end{pmatrix}.$
 		\STATE $Z=P\Pi^T_2Z_R$, $Q = P_2^T Q\left(\begin{smallmatrix}
 		\Pi_1 & \0 \\
 		\0 & \Id_{m-k}
 		\end{smallmatrix}\right)$
 	\end{algorithmic}
 \end{algorithm}
 \vspace{-4mm}
 \subsubsection{Backward error analysis}

 The following theorem provides a backward error bound for the Algorithm \ref{a:COD}.  To our best knowledge, this approach to implementing and analyzing the URV/ULV is new.
 \begin{theorem}
 	Let $k$ be the numerical rank determined in step 2., with the computed trapezoidal (or triangular) factor $\widetilde{R}\in\mathbb{C}^{k\times n}$. 
 	Let $\widetilde{Q}$, $\widetilde{T}_{[11]}\in\mathbb{C}^{k\times k}$ and $\widetilde{Z}$ be the computed factors of complete orthogonal decomposition of $A$. Then there exist backward perturbations $\delta A$, $\Delta A$, $\delta\widetilde{R}$ and exactly unitary matrices $\widehat{Q}\approx \widetilde{Q}$, $\widehat{Z}\approx \widetilde{Z}$ that define the complete orthogonal decomposition 
 	\begin{equation*}
 	A + \delta A + \Delta A + \widehat{Q} \left(\begin{smallmatrix} \delta\widetilde{R} \cr \0 \end{smallmatrix}\right) \widetilde{P}^T = \widehat{Q} \left(\begin{smallmatrix} \Pi_1 & \0 \cr \0 & \Id_{m-k} \end{smallmatrix}\right) \left(\begin{smallmatrix}
 	\widetilde{T}_{[11]} & \0_{k,n-k}\\
 	\0_{m-k,k} & \0 
 	\end{smallmatrix}\right)\widehat{Z}^*\Pi_2 \widetilde{P}^T,
 	\end{equation*}
 	where $\delta A$ is estimated as in (\ref{QRCP-backward-error-columns}), $\Delta A$ is as in (\ref{eq:DA-tilde}), with an estimate as in (\ref{eq:DA-R22}), and 
$ 	\|\delta \widetilde{R}(:,i)\|_2 \leq \varepsilon_3 (1+\varepsilon_1)\|{A}(:,i)\|_2$
 	is bounded analogously to (\ref{eq:dA-rows-1}, \ref{eq:dA-rows-2}).
 \end{theorem}
 \begin{proof}
 	For the first step we have the relation (\ref{QRCP-backward-error-chopped}), where the numerical rank $k$ is determined using a suitable thresholding. 
 	Here it is important to note that $\delta A$ satisfies (\ref{QRCP-backward-error-columns}), and that $\Delta A$ is zero at the $k$ leading pivotal columns selected by $\widetilde{P}$ (see (\ref{eq:DA-tilde})).
 	
 	Set $\widetilde{R}=\begin{pmatrix}
 	\widetilde{R}_{[11]}&\widetilde{R}_{[12]}
 	\end{pmatrix}\in\mathbb{C}^{k\times n}$
 	%
 	and compute the QR factorization of $\widetilde{R}^*$. Here, we deploy complete pivoting for the following reason. We want the overall backward error to be small in each column of $A$ (or at least in the $k$ most important pivotal columns), relative to the norm of that particular column.
To that end, first note that the computed factors $\widetilde{Z}$, $\widetilde{T}^*_{[11]}$ satisfy 
 	\begin{equation}\label{eq:LQ-R}
 	(\widetilde{R} + \delta \widetilde{R})^*\Pi_1 = \Pi_2^T\widehat{Z} \left(\begin{smallmatrix}
 	\widetilde{T}^*_{[11]}\\
 	\0
 	\end{smallmatrix}\right),
 	\end{equation}
where $\widehat{Z}$ is unitary, $\|\widetilde{Z} - \widehat{Z}\|_F \leq \epsilon_2$ and,
 	inserting (\ref{eq:LQ-R}) in (\ref{QRCP-backward-error-chopped}), we obtain 
 	\begin{equation*}
 	(A+\delta A + \Delta A + \widehat{Q} \left(\begin{smallmatrix} \delta\widetilde{R} \cr \0 \end{smallmatrix}\right) \widetilde{P}^T) = \widehat{Q} \left(\begin{smallmatrix} \Pi_1 & \0 \cr \0 & \Id_{m-k} \end{smallmatrix}\right)\left(\begin{smallmatrix}
 	\widetilde{T}_{[11]} & \0\\
 	\0 & \0 
 	\end{smallmatrix}\right)\widehat{Z}^*\Pi_2 \widetilde{P}^T,
 	\end{equation*}
 	where $\delta A$ is from (\ref{QRCP-backward-error}, \ref{QRCP-backward-error-columns}), $\Delta A$ is as in (\ref{eq:DA-tilde}). 
 Now, the key observation is that by (\ref{eq:dA-rows-1}, \ref{eq:dA-rows-2}), we can conclude that $\|\delta \widetilde{R}(:,i)\| \leq \epsilon_3\|\widetilde{R}(:,i)\|$, where, by (\ref{QRCP-backward-error}) and (\ref{QRCP-backward-error-columns}), the column norms of $\widetilde{R}$ are bounded by $\|\widetilde{R}(:,i)\| \leq (1+\epsilon_1) \|A(:,i)\|_2$. 
 	
 	Hence, the $k$ pivotal columns of $A$ (as selected by $\widetilde{P}$) have, individually, small backward errors of the type (\ref{QRCP-backward-error-columns}). Note that the complete pivoting in (\ref{eq:LQ-R}) is essential for column-wise small backward error in $\widetilde{R}$ and thus in $A$.
 \end{proof}
 
 
\subsection{Backward error in rank revealing QR factorizations of $M$ and $K$}
\label{SS=RRQR-Back-err}
The tool of trade in \texttt{quadeig}--type preprocessing is the pivoted QR factorization
\begin{equation}\label{QRCP_M}
(P_{r,M}M) \Pi_M = Q_M R_M,\;\; 
R_M = \left(\begin{smallmatrix} * & * & * & *  \cr
0 & * & * & *  \cr
0 & 0 & * & *  \cr
0 & 0 & 0 & 0  \cr
\end{smallmatrix}\right) = \left(\begin{smallmatrix} \widehat{R}_M \cr \0_{n-r_M,n} \end{smallmatrix}\right) ,
\end{equation}
The initial row sorting before the column pivoted QR factorization is indicated by the matrices $P_{r,M}$, $P_{r,K}$. 
In the absence of row sorting  $P_{r,M}=P_{r,K}=\Id_n$.

In finite precision (see \S \ref{SS=QRF-review})  the computed matrices $\widetilde{Q}_M$, $\widetilde{R}_M$, $\widetilde{\Pi}_M$ satisfy, independent of the choice of the permutation matrix $P_{r,M}$, 
$P_{r,M} (M + \delta M)\widetilde{\Pi}_M = \widehat{Q}_M \widetilde{R}_M$,
where $\widehat{Q}_M=\widetilde{Q}_M+\delta\widetilde{Q}_M$ is exactly unitary with 
\begin{equation}
\|\delta\widetilde{Q}_M\|_F \equiv \|\widehat{Q}_M -\widetilde{Q}_M\|_F \leq \epsilon, \;\mbox{and}\;\; 
\|\delta M(:,i)\|_2 \leq \epsilon_{qr} \|M(:,i)\|_2,\;\;i=1,\ldots, n.
\end{equation}
If $P_{r,M}$ is row sorting permutation, then, in addition, 
$\|\delta M(i,:)\|_2 \leq \epsilon_{qr}^{\rightrightarrows} \|M(i,:)\|_2$, for $i=1,\ldots, n$.
Since $P_{r,M}$ is unitary, we can absorb it into $\widetilde{Q}_M$ and $\widehat{Q}_M$ and redefine  $\widetilde{Q}_M := P_{r,M}^T \widetilde{Q}_M$, 
$\widehat{Q}_M := P_{r,M}^T \widehat{Q}_M$ and write
\begin{equation}\label{eq:QRF(M):backward-2}
(M + \delta M)\widetilde{\Pi}_M = \widehat{Q}_M \widetilde{R}_M
\end{equation} 
\noindent Analogous statements 
hold for the QRF $(P_{r,K}K) \Pi_K = Q_K R_K$, $R_K=\left(\begin{smallmatrix} \widehat{R}_K \cr \0_{n-r_K,n} \end{smallmatrix}\right)$.

\subsubsection{Backward error induced by truncation}\label{SSS=Truncation-Back-Err}
If we truncate the triangular factor in an attempt to infer the numerical rank, we must push the truncated part into the backward error, as in (\ref{eq:A_k:QRCP}). This changes the backward error structure, and the  new error bounds depend on the truncation strategy and the threshold $\tau$.

Assume in (\ref{eq:QRF(M):backward-2}) that we can partition $\widetilde{R}_M$ as 
$$
\widetilde{R}_M = \begin{pmatrix} (\widetilde{R}_M)_{[11]} &  (\widetilde{R}_M)_{[12]} \cr \0_{n-k,k} & (\widetilde{R}_M)_{[22]}\end{pmatrix},\;\;\mbox{where the $(n-k)\times(n-k)$ block  $(\widetilde{R}_M)_{[22]}$ "is small"}.
$$
Then we can write a backward perturbed rank revealing factorization
\begin{equation}\label{eq:M:truncate:backw}
( M + \delta M + \underbrace{\widehat{Q}_M \begin{pmatrix} \0 & \0 \cr \0 & -(\widetilde{R}_M)_{[22]}\end{pmatrix}\widetilde{\Pi}_M^T}_{\Delta M} ) \widetilde{\Pi}_M = \widehat{Q}_M \begin{pmatrix} (\widetilde{R}_M)_{[11]} &  (\widetilde{R}_M)_{[12]} \cr \0_{n-k,k} & \0_{n-k,n-k}\end{pmatrix} .
\end{equation}
Obviously, $\Delta M$ is zero at the $k$ "most linearly independent" columns of $M$ selected by the pivoting, $(\Delta M)\widetilde{\Pi}_M(:,1:k)=\0_{n,k}$. At the remaining $n-k$ columns we have
\begin{equation}
\| (\Delta M)\widetilde{\Pi}_M(:,k+j)\|_2 = \|(\widetilde{R}_M)_{[22]}(:,j) \|_2 \leq \zeta |((\widetilde{R}_M)_{[22]})_{k+1,k+1}|,
\end{equation}
where $\zeta=1+O(n\roff)$ is a correction factor because the pivoting in floating point arithmetic cannot enforce the conditions (\ref{QRCP}) exactly.
Consider the following choices of $k$:
\begin{itemize}
	\item[\framebox{1.}] $k$ is the first index for which $|((\widetilde{R}_M)_{[22]})_{k+1,k+1}| \leq \tau |((\widetilde{R}_M)_{[22]})_{k,k}|$. In that case
	\begin{equation}\label{eq:truncation:back:type1}
	\max_{j=1:n-k}\| (\Delta M)\widetilde{\Pi}_M(:,k+j)\|_2 \leq \zeta \tau |((\widetilde{R}_M)_{[22]})_{k,k}| \leq \zeta \tau \min_{i=1:k}\| (M+\delta M)\widetilde{\Pi}_M(:,i)\|_2 .
	\end{equation}	
	\item[\framebox{2.}] $k$ is the first index for which $|((\widetilde{R}_M)_{[22]})_{k+1,k+1}| \leq \tau\cdot computed(\|M\|_F)$. In that case
	\begin{equation}
	\max_{j=1:n-k}\| (\Delta M)\widetilde{\Pi}_M(:,k+j)\|_2 \leq \tau \cdot computed(\|M\|_F) .
	\end{equation}	
	\item[\framebox{3.}] $k$ is the first index with
	$|((\widetilde{R}_M)_{[22]})_{k+1,k+1}| \leq \tau\cdot computed(\max\{ \|M\|_F, \|C\|_F, \|K\|_F\})$.
	In that case
	$\max_{j=1:n-k}\| (\Delta M)\widetilde{\Pi}_M(:,k+j)\|_2 \leq \tau \cdot computed(\max\{ \|M\|_F, \|C\|_F, \|K\|_F\})$.
	This strategy is used in \texttt{quadeig} with $\tau = n \roff$.
\end{itemize}

\begin{remark}
	{
		If the truncation strategy \boxed{3.}
		is deployed, as in \texttt{quadeig},
		it is necessary to assume that the coefficient matrices have been scaled so that their norms are nearly equal. Otherwise, such a truncation strategy may discard a block in $\widetilde{R}_M$ (and declare $M$ numerically rank deficient) because it is small as compared e.g. to $\|C\|_F$ or $\|K\|_F$. 
		On the other hand, in \texttt{quadeig}, scaling the matrices is optional, and if the (also optional) deflation procedure is enabled, this opens a possibility for catastrophic error (severe underestimate of the numerical ranks) if the matrices are not scaled and if their norms differ by orders of magnitude.  	
	}
\end{remark}

 \section{Kronecker's canonical form for general pencils}\label{s:kronecker}
 Canonical (spectral) structure of a matrix pencil $A-\lambda B$ is, through linearization, an extremely powerful tool for the analysis of quadratic pencils $Q(\lambda) = \lambda^2M + \lambda C + K$. In particular, since the second companion form (\ref{linearization-C2})
 is strong linearization, the partial multiplicities, and thus the structure of all eigenvalues (including infinity), are preserved. In a numerical algorithm for the QEP,  it is desirable to remove the zero and infinite eigenvalues as early as possible and, thus, canonical structure can be used to guide such a preprocessing step.
 
 \subsection{Van Dooren's algorithm}\label{SS=VD-KCF}
 In this section, we briefly review the numerical algorithm by Van Dooren \cite{van1979computation}, developed for the computation of the structure of eigenvalue $\lambda$ for a general $m\times n$ pencil $A-\lambda B$, i.e. the number and the orders of the Jordan blocks for $\lambda$.
%
%
 Here, however, we focus our attention only on computing the structure of the eigenvalue $\lambda=0$. Notice that for the infinite eigenvalue one can reverse the pencil. For an arbitrary finite eigenvalue, a suitably shifted pencil is used; see \cite{van1979computation}.

 For the sake of completeness and later references, we briefly describe the main steps of the staircase reduction for the zero eigenvalue. The pencil $A-\lambda B$ is assumed regular (thus square, $n\times n$), and $\lambda=0$ is assumed to be among its eigenvalues.
 
 \begin{itemize}
 	\item[1.] Compute the singular value decomposition of $A$:
 	$A = U_A\Sigma_AV^*_A$,
 	and let $s_1 = n-\rank(A)$. (Since zero is assumed to be an eigenvalue of $A-\lambda B$, $A$ must be column rank deficient.) Note that $AV_A = \left(\begin{array}{c | c}
 	A_2 & \0_{n\times s_1}
 	\end{array}\right)$, where $A_2$ is of full column rank $n-s_1$. Partition $BV_A = \left(\begin{array}{c | c}
 	B_2 & B_1
 	\end{array}\right)$ in the compatible manner. If we multiply the pencil by $V_A$ from the right we get
 	$(A-\lambda B)V_A = \left(\begin{array}{c | c}
 	A_2-\lambda B_2 & -\lambda B_1
 	\end{array}\right)$.
 	\item [2.] Compute the SVD 
 	$B_1 = U_B\Sigma_B V^*_B,\;\;\mbox{and set}\;\; U^*_B B_1 = \left(\begin{array}{c}
 	B_{1,1} \cr \hline
 	\0_{n-s_1\times s_1}
 	\end{array}\right).$
 	Multiply the pencil $(A-\lambda B)V_A$ by $U^*_B$ from the left to get
 	\begin{equation}\label{pencilUB1}
 	U^*_B(A-\lambda B)V_A = \left(\begin{array}{c | c}
 	A_{2,1} - \lambda B_{2,1} & -\lambda B_{1,1}\cr \hline
 	A_{2,2} - \lambda B_{2,2} & \0_{n-s_1\times s_1}
 	\end{array}\right).
 	\end{equation}
 	
 	\item[3.] Let $P_B$ be the permutation matrix that swaps the row blocks in the above partition. Thus, we have unitary matrices $P_1 = P_BU^*_B$, $Q_1 = V_A$ so that
 	\begin{equation}\label{pencilDeflated1}
 	P_1(A-\lambda B)Q_1 = \left(\begin{array}{c|c}
 	A_{2,2} - \lambda B_{2,2} & \0_{n-s_1\times s_1}\cr \hline
 	A_{2,1} - \lambda B_{2,1} & -\lambda B_{1,1}
 	\end{array}\right),
 	\end{equation}
 	where 
 	$
 	\left(\begin{smallmatrix}
 	A_{2,2} \cr A_{2,1}
 	\end{smallmatrix}\right) = P_1 A_2 \in  \mathbb{C}^{n\times (n-s_1)}
 	$
 	is of full column rank.
 \end{itemize}
 This concludes the first step of the algorithm. Note that he rank of  $B_1$ is $s_1$ (full column rank) since the initial matrix pencil is assumed regular, and
 $$
 |\det P_1 \det(A-\lambda B) \det Q_1| =|\det(A-\lambda B)| =|\lambda|^{s_1}| {|\det B_{1,1}|} |\det(A_{2,2}-\lambda B_{2,2})|,
 $$
 $\det B_{1,1}\neq 0$, which clearly exposes $s_1$ copies of zero in the spectrum, and reduces the problem to the pencil $A_{2,2}-\lambda B_{2,2}$ of lower dimension $n_2=n-s_1$. Also, we see how rank deficiency in $B_{1,1}$ implies singularity of the input pencil. Clearly, if $A_{22}$ is nonsingular, zero has been exhausted from the spectrum of $A-\lambda B$. 
 
 Otherwise, in the next step, we  repeat the described procedure on the $n_2\times n_2$ pencil $A_{2,2}-\lambda B_{2,2}$ to obtain unitary matrices $\widehat{P}_2$, $\widehat{Q}_2$ so that\footnote{Here we abuse the notation slightly. In any block partitioned pencil, the pair of indices $(i,j)$ is interpreted as follows: $i$ is the column index in a block partition, counted from the right to the  left; $j$ is the row index, counted bottom up, $(\leftarrow\mbox{column},\uparrow\mbox{row})$. Each such index pair applies to the current stage in the algorithm. So, $A_{2,1} - \lambda B_{2,1}$ in (\ref{eq:alg-step-3}) is the $(\leftarrow 2,\uparrow 1)$ position in $P_2P_1(A - \lambda B)Q_1Q_2$, which is different from $A_{2,1} - \lambda B_{2,1}$ in (\ref{pencilDeflated1}). This notation reflects the direction of building the final result, from bottom right up.}
 \begin{equation}\label{eq:alg-step-3}
 P_2P_1(A - \lambda B)Q_1Q_2 = \left(\begin{array}{c|c|c}
 A_{3,3} - \lambda B_{3,3} & \0_{n_3\times s_2} & \0_{n_3\times s_1}\cr \hline
 A_{3,2} - \lambda B_{3,2} & -\lambda B_{2,2} & \0_{s_2\times s_1}\cr \hline
 A_{3,1} - \lambda B_{3,1} & A_{2,1} - \lambda B_{2,1} & -\lambda B_{1,1}
 \end{array}\right),
 \end{equation}
 where $P_2 = \diag(\widehat{P}_2, \Id_{s_1}),Q_2 = \diag(\widehat{Q}_2, \Id_{s_1})$, and $s_2=n_2-\mathrm{rank}(A_{22})$, $n_3=n_2-s_2$. As in the first step, $B_{2,2}$ is nonsingular, and $\left(\begin{smallmatrix} A_{3,3}\cr A_{3,2}\end{smallmatrix}\right)$ is of full column rank. Furthermore, since
 $\left(\begin{smallmatrix} \0_{n_3+s_2,s_2}\cr A_{2,1}\end{smallmatrix}\right)$ is a column block in the full column rank matrix, $A_{2,1}$ must have full column rank as well.
  %
 This procedure is repeated until in an $\ell$th step we obtain 
 
 
  \begin{equation} \label{eq:SVDKronecker}
  P(A-\lambda B)Q = \left(\begin{array}{c |c|c|c|c}
  A_{\ell+1,\ell+1} - \lambda B_{\ell+1,\ell+1} & \0_{n_{\ell+1}\times s_\ell} & \ldots & \0_{n_{\ell+1}\times s_2} & \0_{n_{\ell+1}\times s_1}\cr \hline
  A_{\ell+1,\ell} - \lambda B_{\ell+1,\ell} & -\lambda B_{\ell,\ell} & \ldots & \0_{s_{\ell}\times s_2} & \0_{s_{\ell}\times s_1} \cr \hline
  \vdots & \vdots & \ddots & \vdots & \vdots \cr \hline
  A_{\ell+1,2}-\lambda B_{\ell+1,2} & A_{\ell,2} -\lambda B_{\ell,2} & \ddots & -\lambda B_{2,2} & \0_{s_2,s_1}\cr \hline
  A_{\ell+1,1}-\lambda B_{\ell+1,1} & A_{\ell,1} -\lambda B_{\ell,1} & \ddots & A_{2,1}-\lambda B_{2,1} & -\lambda B_{1,1} 
  \end{array}\right),
  \end{equation}
 with nonsingular $A_{\ell+1,\ell+1}$. Then, by design,  $B_{i,i}$ has full rank $s_i$ for $i=1,\ldots,\ell$, and $A_{i,i-1}$ has full column rank $s_i$ for $i=2,\ldots,\ell$.
See  \cite[Algorithm 3.1]{van1979computation}.
%
%
 \begin{proposition}\cite[Proposition 3.5]{van1979computation}\label{t:KroneckerZero}
 	The indicies $s_j$ given by Algorithm \cite[Algorithm 3.1]{van1979computation} completely determine the structure at $\lambda=0$ of the pencil $A-\lambda B$, i.e. $A-\lambda B$ has $s_j-s_{j+1}$ elementary divisors $\lambda^j$, $j=1,\ldots,\ell$.
 \end{proposition}
\noindent Finally, we can conclude that this algorithm also determines the structure of zero eigenvalue for the quadratic eigenvalue problem via a (strong) linearization.
 \begin{theorem}\label{t:deflation}
 	If 
 	\cite[Algorithm 3.1]{van1979computation} is  applied to the pencil (\ref{linearization-C2}), then it completely determines the structure of the eigenvalue zero for the quadratic eigenvalue problem $Q(\lambda)x \equiv  (\lambda^2M + \lambda C + K)x=\0$.
 \end{theorem}

\subsection{Computing Kronecker's Canonical form using rank revealing QR factorization}\label{SS=KCF-QR}
It is clear that the key rank reveling in the Van Dooren's algorithm can be done using the rank revealing QR factorization instead of the SVD. In \S \ref{SS=QRF-review} we argued that in practice rank revealing QR factorization is considered as reliable as the SVD, and that it may even have some advantages.
If we replace the SVD $A = U_A \Sigma_A V_A^*$ with $A = Q_A R_A P_A^T$, then the post-multiplication  by $V_A$ in the procedure described in \S \ref{SS=VD-KCF} is replaced with column permutation, i.e. post-multiplication with the permutation matrix $P_A$, which is error free even in finite precision arithmetic. Also, we can rearrange the elimination process to obtain an upper block triangular matrix as follows.


\begin{itemize}
	\item[1.] Compute the rank revealing factorization of $A_{1,1} = A$
	$A_{1,1}P_A = Q_AR_A$,
	and denote $s_1 = n_1- \rank(A) = n - \rank(A)$. Now, $Q_A^*A_{1,1} = \left(\begin{array}{c}
	A_2\cr \hline
	\0_{s_1\times n}
	\end{array}\right)$. Partition $Q^*_AB = \left(\begin{array}{c}
	B_2\cr \hline
	B_1
	\end{array}\right)$	in compatible manner. Multiply the pencil $(A-\lambda B)$ by $Q^*_A$ on the left to get
	\begin{equation}\label{eq:rrQRKornecker1}
	Q^*_A(A-\lambda B) = \left(\begin{array}{c}
	A_2 - \lambda B_2\cr \hline
	\lambda B_1
	\end{array}\right).
	\end{equation}
	\item[2.] Compute complete orthogonal decomposition of $B_1$, 
	$B_1 = U_BR_BV^*_B$.
	The column rank of $B_1$ is $s_1$, if the matrix pencil is regular, and $B_1V_B = \left(\begin{array}{c| c}
	B_{1,1} & \0_{s_1,n-s_1}
	\end{array}\right)$, where $B_{1,1}$ is upper triangular. Multiply the pencil (\ref{eq:rrQRKornecker1}) by $V_B$ on the right to get
	$Q_A^*(A-\lambda B)V_B = \left(\begin{array}{c| c}
	A_{1,2} - \lambda B_{1,2} & A_{2,2} - \lambda B_{2,2}\cr \hline
	\lambda B_{1,1} & \0
	\end{array}\right)$.
	\item[3.] Let $P_B$ be the permutation matrix for permuting the $s_1$ and $n-s_1$ column blocks. Define $P_1 = Q^*_A$ and $Q_1 = V_BP_B$. The first Jordan block for the eigenvalue $\lambda=0$ is deflated by the following orthogonal transformation:
	\begin{equation}\label{eq:rrQRKronecker2}
	P_1(A-\lambda B)Q_1 = \left(\begin{array}{c |c}
	A_{2,2} -\lambda B_{2,2} & A_{1,2} - \lambda B_{1,2}\cr \hline
	\0 & \lambda B_{1,1}
	\end{array}\right).
	\end{equation}
\end{itemize}
This reduction continues analogously to \S \ref{SS=VD-KCF}, but keeping the upper triangular structure; the details are in Algorithm \ref{a:Deflation0}. 
\begin{algorithm}[h!]
	\caption{Deflation of the eigenvalue $\lambda=0$ of the $n\times n$ pencil $A-\lambda B$}
	\label{a:Deflation0}
	\begin{algorithmic}[1]
		\STATE $j=1;$ $A_{1,1} = A;$ $B_{1,1} = B;$ $n_1 = n;$\;
		\STATE $s_1=n-\mathrm{rank}(A_{1,1})$
		\WHILE{$s_i>0$}
		\STATE Compute rank revealing QR: $A_{j,j}P_A = Q_AR_A$\;
		\STATE Partition matrices: $\left(\begin{array}{c}
		A_{j+1}\\
		\0
		\end{array}\right) = Q^*_AA_{j,j}$, $\left(\begin{array}{c}
		B_{j+1}\\
		B_{j}
		\end{array}\right) = Q^*_AB_{j,j}$\;
		\STATE Update and partition the blocks in the row $j$
		\FOR{$i=1:j-1$}
		\STATE $\left(\begin{array}{c}
		A_{i,j+1} \\
		A_{i,j}
		\end{array}\right) = Q^*_AA_{i,j};$ $\left(\begin{array}{c}
		B_{i,j+1} \\
		B_{i,j}
		\end{array}\right) = Q^*_AB_{i,j};$
		\ENDFOR
		\STATE Compute complete orthogonal decomposition of the $s_j \times n_j$ matrix $B_j$: $B_j = A_BR_BV^*_B$\;
		\STATE Compress $B_j$ to full column rank, permute and parititon:
		\STATE $\left(\begin{array}{c c}
		A_{j+1,j+1} & A_{j,j+1}
		\end{array}\right) = A_{j+1}V_BP_B;$ $\left(\begin{array}{c c}
		B_{j+1,j+1} & B_{j,j+1}
		\end{array}\right) = B_{j+1}V_BP_B$\;
		\STATE $\left(\begin{array}{c c}
		\0 & B_{j,j}
		\end{array}\right) = B_{j}V_BP_B$\;
		\STATE $n_{j+1} = n_j-s_j$, $j=j+1$
		\STATE Compute the rank revealing QR factorization $A_{j,j}P_A = Q_A R_A$
		\STATE $s_j = n_j-\mathrm{rank}(A_{j,j})$
		\ENDWHILE
	\end{algorithmic}
\end{algorithm}
The final form of $P(A-\lambda B)Q$ is given in (\ref{eq:SVDKroneckerr}). The deflation of infinite eigenvalues can be done by the same algorithm, but with the reversed pencil $B-\lambda A$.
\begin{equation}\label{eq:SVDKroneckerr}
 \left(\begin{array}{c| c| c| c| c}
A_{\ell+1,\ell+1} - \lambda B_{\ell+1,\ell+1} & A_{\ell,\ell+1} - \lambda B_{\ell,\ell+1} & \ldots & A_{2,\ell+1} - \lambda B_{2,\ell+1} & A_{1,\ell+1} -\lambda B_{1,\ell+1}\cr \hline
\0 &  -\lambda B_{\ell,\ell} & \ldots & A_{2,\ell} - \lambda B_{2,\ell} & A_{1,\ell} -\lambda B_{1,\ell}\cr \hline
\vdots & \vdots & \ddots & \vdots & \vdots \cr \hline
\0 & \0 & \ldots & -\lambda B_{2,2} & A_{1,2}-\lambda B_{1,2}\cr \hline
\0 & \0 & \ldots & \0 & -\lambda B_{1,1}
\end{array}\right)
\end{equation}
\begin{remark}
For stronger backward stability, the rank revealing QR factorizations  can be computed with full pivoting, e.g. $P_{A_r}A_{j,j}P_A = Q_A R_A$. The rest of the algorithm is changed simply by setting $Q_A\leftarrow P_{A,r}^T Q_A$, cf. Remark \ref{REM:QRCPP}. The complete orthogonal decomposition is computed as described in \S \ref{ss:COD}.
\end{remark}

\section{Scaling and balancing}\label{S=S+B}
Backward stable computation of the eigenvalues of the linearized pencil (\ref{linearization-C2}) allows interpretation of the computed output as the exact spectral information of a pencil $A+\delta A - \lambda (B+\delta B)$, where $\delta A$ ($\delta B$) is small relative to $A$ ($B$) in an appropriate norm. Since $\delta A$, $\delta B$ are not structured (linear eigensolver that takes $A$ and $B$ on input is oblivious of the block structure and the original quadratic problem), the backward stability does not necessarily hold in terms of the original data $M$, $C$, $K$. In fact, it is easily seen that one potential source of the problem is when there is large variation in the norms of $M$, $C$, $K$, $\Id_n$. This issue is addressed by parameter scaling.

\subsection{Parameter scaling}
 
The idea is to define a scaled problem $\widetilde{Q}(\mu) = \mu^2 \widetilde{M} + \mu \widetilde{C} + \widetilde{K}$, which depends on two scalar parameters $\gamma$ and $\delta$, where
	$\lambda = \gamma \mu, \;\; \widetilde{Q}(\mu) = Q(\lambda)\delta = \mu^2 (\gamma^2\delta M) + \mu (\gamma \delta C) + \delta K$.
This scaling has no effect on backward error for QEP, but only on backward error for the linearization. In \texttt{quadeig}, \cite{Hammarling:QUADEIG} used two types of scaling.
\paragraph{\underline{Fan, Lin and Van Dooren scaling}}\cite{Fan-Lin-Dooren-Scaling} The scaling parameters $\gamma$ and $\delta$ are defined as the solution of minimization problem
	$\min_{\gamma,\delta} \max \{\|\widetilde{K}\|_2\!-\! 1,\|\widetilde{C}\|_2\!-\!1,\|\widetilde{M}\|_2\!-\! 1\}$,
so that in the linearization (\ref{linearization-C2}) of $\widetilde{Q}(\mu)$ all nonzero blocks are of norm close to one. The solution of this minimization problem is 
	$\gamma = \sqrt{{\|K\|_2}/{\|M\|_2}}$, $\delta = {2}/{(\|K\|_2+\|C\|_2\gamma)}$.
\paragraph{\underline{Tropical scaling}} $\gamma$ and $\delta$ are defined as tropical roots of max-times scalar quadratic polynomial
$
	q_{\text{trop}}(x) = \max (\|M\|_2x^2,\|C\|_2x,\|K\|_2), \;\; x \in [0,\infty\rangle.
$
Define $\tau_Q = \frac{\|C\|_2}{\sqrt{\|M\|_2\|K\|_2}}$. If $\tau_Q \leq 1$, there is double root
	$\gamma^+ = \gamma^- = \sqrt{\|K\|_2/\|M\|_2}$,
and if $\tau_Q>1$ there are two distinct roots
	$\gamma^+ = {\|C\|_2}/{\|M\|_2} > \gamma^- = {\|K\|_2}/{\|C\|_2}$.
Hence, when $\tau_Q>1$, scaling with 
	$\gamma = \gamma^+,\;\;\delta = (q_{\text{trop}}(\gamma^+))^{-1}$
is used to compute the eigenvalues outside of the unit circle, and scaling with
$\gamma = \gamma^-,\;\;\delta = (q_{\text{trop}}(\gamma^-))^{-1}$
 is used to compute the eigenvalue inside the unit circle.
For further details on parameter scaling see \cite{Gaubert-tropical}. 
\subsection{Balancing}\label{SS=B}
Backward error $\delta A$, such that $\|\delta A\|/\|A\|$ is small in some appropriate matrix norm may not guarantee high accuracy of the result (e.g. of the computed eigenvalues) if  the corresponding condition number is too big. Also, the interpretation of the computed result in the backward error analysis framework may be difficult because small $\|\delta A\|/\|A\|$ does not guarantee that the backward relative error in the individual matrix entries is correspondingly small. This problem in particular escalates if the matrix entries vary over several orders of magnitude, and the smallest entries carry relevant information. Well known technique to cope with the unbalanced matrix entries is the technique of balancing: replace $A$ with a similar matrix $D^{-1} A D$, where $D$ is a  diagonal matrix (or a permutation of diagonal matrix) that in some sense balances the matrix entries. For instance, for the non-symmetric eigenvalue problem, the balancing strategy proposed by \cite{Osborne-balancing}, \cite{Parlett-Reinsch-balance} ensures that for all $i$, the $i$th row and the $i$th column of $D^{-1} A D$ have (nearly) the same norm; usually the scaling factors $D_{ii}$ are replaced with closest power of the base of the machine arithmetic, so  that the balancing step only shifts the exponents of the matrix entries, leaving the mantissas unchanged. For a software implementation see \cite{2014arXiv1401.5766J}. Another interesting strategy of  is to make $D^{-1}AD$ closer to normal matrices \cite{Lemonnier:2006}. 
It should be noted that balancing may be harmful, see \cite{Watkins-balancing}.


Balancing has been recognized as an important preprocessing technique for eigensolvers for general pencils $A-\lambda B$, see e.g.  \cite{ward1981balancing}, \cite{Lemonnier:2006}. In some applications, it is convenient to apply it to matrix triplets. For instance, the frequency response matrix $C(\sigma E - A)^{-1}B$ of a MIMO descriptor linear system is efficiently computed for many values of $\sigma$ by using precomputed $m$-Hessenberg-triangular form of the triplet $(E,A,B)$. It is shown in \cite{bosner2014balancing} that the numerical robustness can be improved by pre-scaling $D_lAD_r$, $D_lED_r$ and $D_lB$, where the diagonal matrices $D_l=\mathrm{diag}(10^{l_i})_{i=1}^n$ and $D_r=\mathrm{diag}(10^{r_i})_{i=1}^n$ are such that the range of magnitude orders of all elements in the matrices $D_lAD_r$, $D_lED_r$ and $D_lB$ is minimal.


We will show that balancing can substantially improve numerical solution of the QEP, and that it should be considered as standard preprocessing technique, together with the parameter scaling. Balancing with $D_l$ and $D_r$ yields an equivalent QEP:
\begin{equation}\label{eq:diagScaledQEP}
	\widehat{Q}(\lambda) = \lambda^2(D_l M D_r) + \lambda(D_l C D_r) + (D_l K D_r).
\end{equation}
For a computed eigenpair $(\lambda,x)$ of (\ref{eq:diagScaledQEP}), the corresponding right eigenpair for the original problem is $(\lambda,D_r x)$. Note that this corresponds to the transformation of $A-\lambda B$ of the linearization (\ref{linearization-C2}) to 
$
(D_l \oplus D_l) (A - \lambda B) (D_r \oplus D_l^{-1}).
$

The scaling changes the matrix entries so that e.g. $m_{ij} \mapsto 10^{l_i} m_{ij} 10^{r_j}$, and logarithmic magnitude of the scaled element is $\log_{10} 10^{l_i} |m_{ij}| 10^{r_j} = l_i + r_j + \log_{10}|m_{ij}|$.
The idea of balancing is to make these as close to zero as possible over all entries of all matrices, where the importance of each matrix is weighted with a corresponding factor $\alpha_M\geq 0$, $\alpha_C\geq 0$, $\alpha_K\geq 0$, respectively. This yields the minimization problem $\min_{l=(l_i),r=(r_i)\in \mathbb{R}^{n}} \varphi(l,r)$, where
\begin{equation}\label{eq:ScalingProblem}
\varphi(l,r)\! =\!\!\!
 \sum^{n}_{i=1} [ \alpha_M \!\!\!\!\! \sum^n_{\substack{j=1\\m_{ij}\neq 0}}\!\!\!(l_i+r_j+\log_{10}\!|m_{ij}|)^2
+  \alpha_C \!\!\!\!\sum^n_{\substack{j=1\\c_{ij}\neq 0}}\!\!(l_i+r_j+\log_{10}\!|c_{ij}|)^2 +\alpha_K\!\!\!\! \sum^n_{\substack{j=1\\k_{ij}\neq 0}}\!\!(l_i+r_j+\log_{10}\!|k_{ij}|)^2],
\end{equation}
which can be written in the generic form $\|A x - b\|_2\rightarrow\min$ as  follows.
\begin{proposition}
For each index $i$ define row vectors
$\mathbbm{1}^{\alpha_M}_{i} = \alpha_M \mathcal{I}(M(i,:))$, 
$\mathbbm{1}^{\alpha_C}_{i} = \alpha_C \mathcal{I}(C(i,:))$, $\mathbbm{1}^{\alpha_K}_{i} = \alpha_K \mathcal{I}(K(i,:))$, where 
$\mathcal{I}(\cdot)$ is the indicator vector, e.g. $(\mathcal{I}(M(i,:)))_j=1$ if $m_{ij}\neq 0$ and $0$ otherwise.
Further, define diagonal matrices
$\Id^{\alpha_M}_i=\mathrm{diag}(\mathbbm{1}^{\alpha_M}_{i})$, $\Id^{\alpha_C}_i=\mathrm{diag}(\mathbbm{1}^{\alpha_C}_{i})$, $\Id^{\alpha_K}_i=\mathrm{diag}(\mathbbm{1}^{\alpha_K}_{i})$
and column vectors in $\mathbb{R}^n$
$L^{\alpha_M}_i = \alpha_M [\log|M(i,:)|\circ\mathcal{I}(M(i,:))]^T$, 
$L^{\alpha_M}_i = \alpha_M [\log|M(i,:)|\circ\mathcal{I}(M(i,:))]^T$, 
$L^{\alpha_M}_i = \alpha_M [\log|M(i,:)|\circ\mathcal{I}(M(i,:))]^T$,
where $\circ$ denotes the Hadamard product and $0\log 0$ is set to $0$.
Then \eqref{eq:ScalingProblem} is a linear least squares problem $\|Ax-b\|_2\to \text{min}$ with
\begin{equation}
A = \left(\begin{array}{c|c}
e_1\otimes \mathbbm{1}^{\alpha_M}_{1} & \Id^{\alpha_M}_1\\
\vdots & \vdots\\
e_n\otimes \mathbbm{1}^{\alpha_M}_{n} & \Id^{\alpha_M}_n\\ \hline
e_1\otimes \mathbbm{1}^{\alpha_C}_{1} & \Id^{\alpha_C}_1\\
\vdots & \vdots\\
e_n\otimes \mathbbm{1}^{\alpha_C}_{n} & \Id^{\alpha_C}_n\\ \hline
e_1\otimes \mathbbm{1}^{\alpha_K}_{1} & \Id^{\alpha_K}_1\\
\vdots & \vdots\\
e_n\otimes \mathbbm{1}^{\alpha_K}_{n} & \Id^{\alpha_K}_n
\end{array}\right),\;\;\; x = \begin{pmatrix}
l\\
r
\end{pmatrix},\;\;\; b = \left(\begin{array}{c}
L^{\alpha_M}_1\\
\vdots\\
L^{\alpha_M}_n\\\hline
L^{\alpha_C}_1\\
\vdots\\
L^{\alpha_C}_n\\\hline
L^{\alpha_K}_1\\
\vdots\\
L^{\alpha_K}_n
\end{array}\right).\;\;(\mbox{$\otimes$ is the Kronecker product.})
\end{equation}
The corresponding system  of normal equations, $Lx = p$, reads
\begin{equation*}
L = \begin{pmatrix}
F_1 & G\\
G^T & F_2
\end{pmatrix}, \;\; p = \begin{pmatrix}
-c\\
-d
\end{pmatrix},\;\; x = \begin{pmatrix}
l\\
r
\end{pmatrix},\;\begin{array}{l}
F_1 = \diag(n_{r_1},\ldots,n_{r_n}) \in \mathbb{R}^{n\times n}\cr
F_2 = \diag(n_{c_1},\ldots,n_{c_n}) \in \mathbb{R}^{n\times n}
\end{array}
\end{equation*}
with
$
n_{r_i}\!\! =\!\! \sum^n_{\substack{j=1\\m_{ij}\neq 0}}\alpha_M +\sum^n_{\substack{j=1\\c_{ij}\neq 0}}\alpha_C + \sum^n_{\substack{j=1\\k_{ij}\neq 0}}\alpha_K, 
n_{c_j}\!\! =\!\! \sum^n_{\substack{i=1\\m_{ij}\neq 0}}\alpha_M +\sum^n_{\substack{i=1\\c_{ij}\neq 0}}\alpha_C + \sum^n_{\substack{i=1\\k_{ij}\neq 0}}\alpha_K;
$
$G\in \mathbb{R}^{n\times n}$ is the sum of incidence matrices of $M$, $C$ and $K$ scaled by $\alpha_M,\alpha_C$ and $\alpha_K$: $G=\alpha_M \mathcal{I}(M)+\alpha_C \mathcal{I}(C)+\mathcal{I}(K)$, where e.g. $\mathcal{I}(M)_{ij}=1$ if $m_{ij}\neq 0$ and $0$ otherwise.
The entries of $c\in \mathbb{R}^n$ are
$c_i = \alpha_M\sum^n_{\substack{j=1\\m_{ij}\neq 0}}\log|m_{ij}| + \alpha_C\sum^n_{\substack{j=1\\c_{ij}\neq 0}}\log|c_{ij}| + \alpha_K\sum^n_{\substack{j=1\\k_{ij}\neq 0}}\log|k_{ij}|$,
and of $d\in \mathbb{R}^n$ are
$d_j = \alpha_M\sum^n_{\substack{i=1\\m_{ij}\neq 0}}\log|m_{ij}| + \alpha_C\sum^n_{\substack{i=1\\c_{ij}\neq 0}}\log|c_{ij}| + \alpha_K\sum^n_{\substack{i=1\\k_{ij}\neq 0}}\log|k_{ij}|$.
\end{proposition}
In the case of large scale sparse problem, the normal equations can be solved using the preconditioned conjugate gradients method, see e.g.  \cite{bosner2014balancing}. 

In Figure \ref{fig:beam-cbe-right}, we illustrate the impact of balancing to the backward error in the case of one NLEVP benchmark example (damped beam). In general, this kind of preprocessing should also be used in the iterative methods. 

	\begin{figure}[H]
		\centering
		\includegraphics[width=0.45\textwidth, height=1.6in]{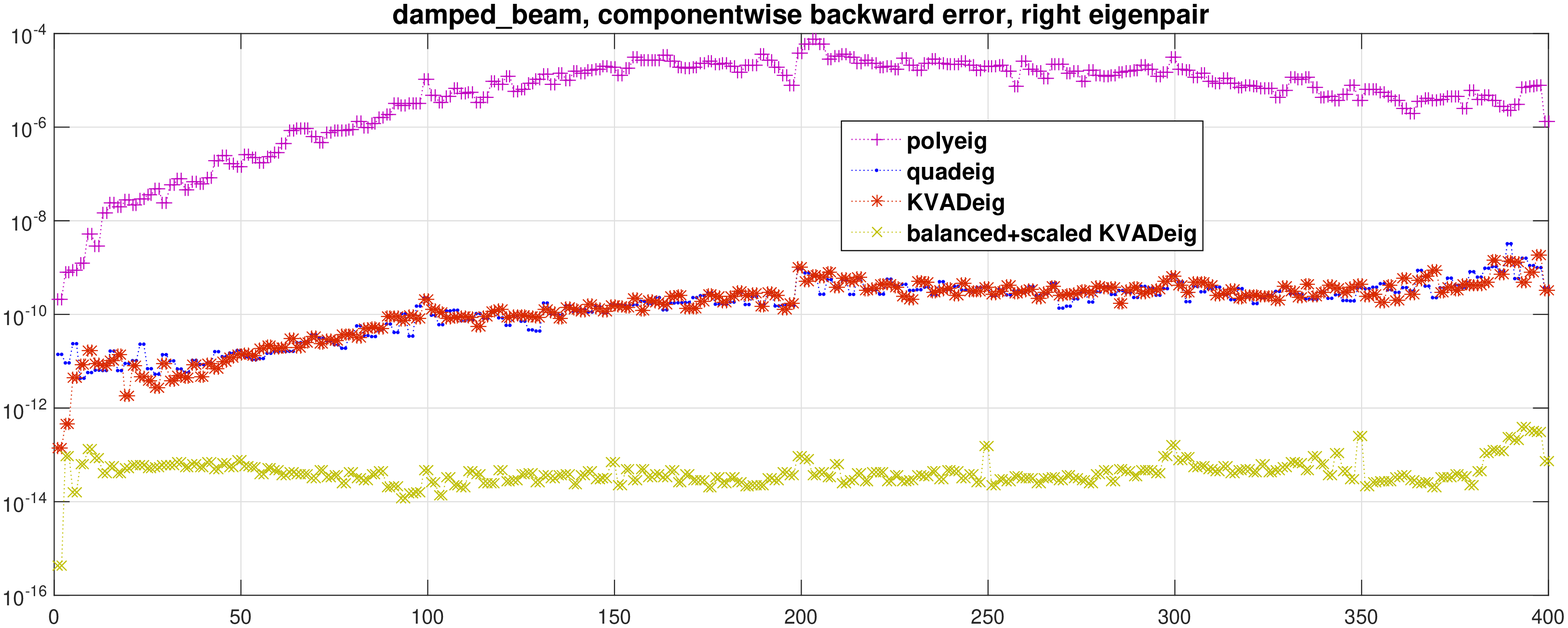}
		\includegraphics[width=0.45\textwidth, height=1.6in]{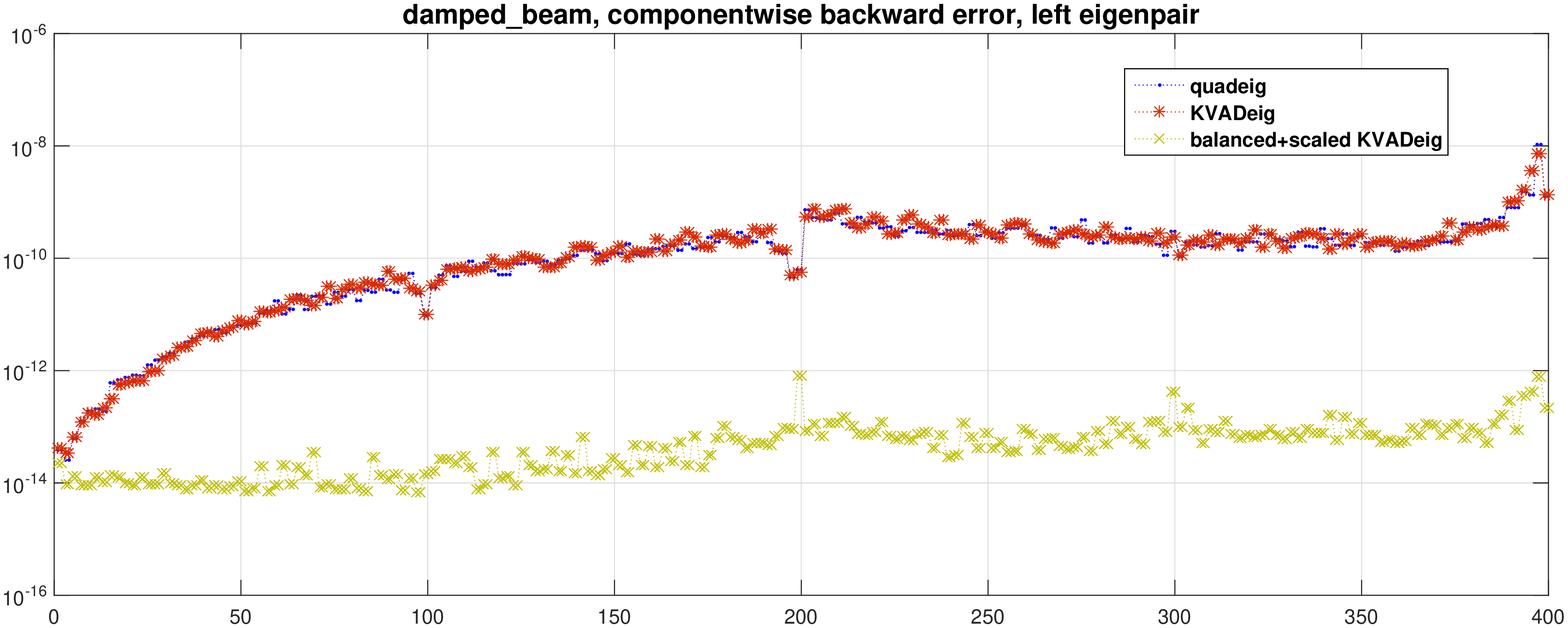}
		\caption{\label{fig:beam-cbe-right} The example \texttt{damped\_beam} from the NLEVP collection. The diagonal balancing included in the new algorithm \texttt{KVADeig} helps in reducing the component-wise backward error of the \texttt{quadeig} algorithm.}	
	\end{figure}

%
\section{Improved deflation process}\label{S=New-Deflation}

In this section, we  propose an extension of the \texttt{quadeig} reduction scheme toward the KCF outlined in \S \ref{SS=KCF-QR}. To motivate, we use an example from the NLEVP collection.

\subsection{A case study example} 
This is a $10\times 10$ quadratic eigenvalue problem for the pencil $\mathcal{I}(\lambda)=\lambda^2 M + \lambda C + K$, whose real eigenvalues and the corresponding eigenvectors give the intersection points of a sphere, a cylinder and a plane. 
Although of small dimension and very simple structure, this example is an excellent illustration of difficulties in solving nonlinear eigenvalue problems.

It has been shown in \cite{morgan1992polynomial}, \cite{Manocha1994SolvingSO} that $\mathcal{I}(\lambda)$ has only four finite eigenvalues: two real ones and a complex conjugate pair. 
We take this example as a case study and compute the spectrum by several mathematically equivalent methods.
If one plainly applies the QZ to a linearization of $\mathcal{I}(\lambda)$, such as the first or the second companion form with the Fan-Lin-Van Dooren scaling, the spectrum appears as
\begin{eqnarray}
\!\!\!\!\!\!&& \mathcal{C}_1(\lambda)\! : \! \left\{
\begin{array}{l}
\lambda_1 = \texttt{2.476851749893558e+01},\;\;
\lambda_2 = \texttt{2.476851768196165e+01} \cr
\lambda_3 = \texttt{-5.581844429198920e+08 - 1.628033679447590e+09} \ii \cr 
\lambda_4 = \texttt{-5.581844429198920e+08 + 1.628033679447590e+09} \ii \cr 
\lambda_5 = \texttt{2.570601782117493e+18}, \;\;
\lambda_6 =\ldots = \lambda_{14}=\texttt{Inf},\; \lambda_{15}=\ldots=\lambda_{20}=\texttt{-Inf}
\end{array} \right. \label{eq:intersect:C1-scaled}\\ 
\!\!\!\!\!\!&& \mathcal{C}_2(\lambda)\! : \! \left\{
\begin{array}{l}
\lambda_1 = \texttt{2.476851749893561e+01}, \;\;
\lambda_2 = \texttt{2.476851768196167e+01} \cr
\lambda_3 = \texttt{-2.653302084597818e+09}, 
\lambda_4 =\ldots = \lambda_{17}=\texttt{Inf},  \lambda_{18}\! =\ldots=\lambda_{20}\!=\texttt{-Inf}
\end{array} \right.\label{eq:intersect:C2-scaled}
\end{eqnarray}
If we use the same method, but with the reversed pencil $\mu^2 K +\mu C + M$, ($\lambda=1/\mu$)  then from the first companion form QZ has computed $12$ finite eigenvalues, and from the second $10$. 
If we run the Matlab's solver \texttt{polyeig()}, we obtain 
\begin{equation}\label{eq:intersect:polyeig}
\texttt{polyeig}(\mathcal{Y}(\lambda))\; : \; \left\{
\begin{array}{l}
\lambda_1 = \texttt{2.476851768196161e+01}, \;\
\lambda_2 = \texttt{2.476851749893561e+01}, \cr
\lambda_3 = \texttt{1.426603361688555e+08},\;\; 
\lambda_4 = \texttt{-1.353812777123886e+08} \cr 
\lambda_5 =\ldots = \lambda_{18}=\texttt{Inf},\; \lambda_{19}=\lambda_{20}=\texttt{-Inf}
\end{array} \right.
\end{equation}
and if we scale the coefficient matrices then 
\begin{equation}\label{eq:intersect:polyeig-scaled}
\texttt{polyeig}(\mathcal{Y}_{scaled}(\lambda))\; : \; \left\{
\begin{array}{l}
\lambda_1 = \texttt{2.476851768196165e+01}, \;\;
\lambda_2 = \texttt{2.476851749893559e+01} \cr
\lambda_3 = \texttt{-3.020295324523709e+08 + 1.229442619245432e+09} \ii \cr 
\lambda_4 = \texttt{-3.020295324523709e+08 - 1.229442619245432e+09} \ii \cr 
\lambda_5 =\ldots = \lambda_{18}=\texttt{Inf},\; \lambda_{19}=\lambda_{20}=\texttt{-Inf}
\end{array} \right.
\end{equation}
Almost perfect match in $\lambda_1$ and $\lambda_2$ is reassuring, but there is an obvious  disagreement in the total number and the nature (real or complex) of finite eigenvalues. With an earlier version of Matlab, the results that correspond to (\ref{eq:intersect:C1-scaled}), (\ref{eq:intersect:C2-scaled}), (\ref{eq:intersect:polyeig}),  and (\ref{eq:intersect:polyeig-scaled}) coincide in the numbers of finite eigenvalues; $\lambda_1$ and $\lambda_2$ are close up to machine precision, but the remaining computed finite eigenvalues are substantially different.

The rank of the matrix $M$ is exactly $3$, and it will be correctly determined numerically due to a particularly simple sparsity structure of $M$. The matrix $K$ is also sparse with $\kappa_2(K) \approx \texttt{4.09+03}$, so there is no numerical rank issue. In this situation, a preprocessing procedure such as in \texttt{quadeig} will reverse the pencil and deflate $7$ zero eigenvalues (infinite eigenvalues of the original problem) at the very beginning. The remaining eigenvalues are then computed (e.g. using \texttt{quadeig}) as\footnote{Here, to save the space, we display the computed values only to five digits.} 
\begin{equation}
\begin{array}{l||l}
\lambda_1 = \texttt{2.4769e+001}   &        \lambda_8 = \texttt{-1.4660e+007 - 6.9064e+006} \ii \cr 
\lambda_2 = \texttt{2.4769e+001}   &   \lambda_9 =  \texttt{-1.4660e+007 + 6.9064e+006} \ii \cr       	
\lambda_3 = \texttt{1.1194e+006}   &  \lambda_{10} =  \texttt{-4.5822e+015} \cr 
\lambda_4 = \texttt{-5.5674e+005 -1.0143e+006} \ii & \lambda_{11} = \texttt{-3.9134e+015} \cr 
\lambda_5 = \texttt{-5.5674e+005 + 1.0143e+006} \ii & \lambda_{12} = \texttt{-2.3047e+019} \cr 
\lambda_6 = \texttt{1.4679e+007 - 1.9395e+007} \ii  & \lambda_{13}= \texttt{3.0862e+020}\cr
\lambda_7 = \texttt{1.4679e+007 + 1.9395e+007} \ii  &
\end{array}
\end{equation}
After the deflation of the $7$ zero eigenvalues, in the thus obtained linearization $A-\lambda B$, the rank of the matrix $A$ is $7$, and it can be determined exactly because of sparsity ($A$ has $6$ zero columns, and the remaining $7$ ones build a well conditioned $13\times 7$ submatrix of $A$). The matrix $B$ is well conditioned. 
This means that at least $6$ more zero eigenvalues are present in the reversed problem (infinities in the original problem); those are not detected by the QZ algorithm running on $A-\lambda B$.
\begin{remark}
		It should be noted that the successful removal of many infinite eigenvalues in 
		(\ref{eq:intersect:C1-scaled}), (\ref{eq:intersect:C2-scaled}), (\ref{eq:intersect:polyeig}),  and (\ref{eq:intersect:polyeig-scaled}) is due to the sparsity of $M$, $C$, $K$ (and the linearization $A-\lambda B$) that is successfully exploited by the preprocessing in the QZ algorithm. Recall, before the reduction to the triangular - Hessenberg form the matrices are scaled and permuted, as described in \cite{ward1981balancing} in order to get an equivalent pencil $\hat{A} - \lambda \hat{B}$ with the block structure
		\begin{equation}
			\hat{A} = \left(\begin{smallmatrix}
			A_{11} & A_{12}D_2G_2 & A_{13}\\
			0 & G_1D_1A_{22}D_2G_2 & G_1D_1A_{23}\\
			0 & 0 & A_{33}
			\end{smallmatrix}\right),\;\; \hat{B} = \left(\begin{smallmatrix}
			B_{11} & B_{12}D_2G_2 & B_{13}\\
			0 & G_1D_1B_{22}D_2G_2 & G_1D_1B_{23}\\
			0 & 0 & B_{33}
			\end{smallmatrix}\right),
		\end{equation}
		where $A_{11}, A_{33}, B_{11}, B_{33}$ are upper triangular, and
the matrices $D_1$ and $D_2$ are computed so that the elements of $D_1A_{22}D_2$ and $D_1B_{22}D_2$ have magnitudes close to one. Further, $G_2$ is the permutation matrix determined so that the ratios of the column norms of $D_1A_{22}D_2G_{2}$ to the corresponding column norms of $D_1B_{22}D_2G_2$ appear in decreasing order;  $G_1$ is determined so that the ratios of row norms of $G_1D_1A_{22}D_2G_{2}$ to those of $G_1D_1B_{22}D_2G_{2}$ appear in decreasing order. See also \cite{kresner-thesis}.
		
	\end{remark}



\subsection{Deflation process revisited -- a modified scheme}\label{SS=New-Reduce-KCF}
Let us now get back to the initial idea of \texttt{quadeig} -- removing the zero and the infinite eigenvalues in the preprocessing phase of the computation. To introduce a new scheme, consider removing the zeros from the spectrum and assume that a rank revealing QR factorization 
$
K P_K = Q_K \left(\begin{smallmatrix} R_{K,1} \cr \0_{n-r_K,n}\end{smallmatrix}\right)=Q_{K,1}R_{K,1},
$
$Q_K=(Q_{K,1}, Q_{K,2})$,
has detected that $K$ is numerically rank deficient. Consider the transformation
\begin{eqnarray}
	\mathbf{P}_1(A-\lambda B)\mathbf{Q}_1 &=& \left(\begin{array}{c |||c}
	Q^*_K & \0 \cr \hline\hline\hline
	\0 & Q^*_K
	\end{array}\right) (\begin{pmatrix} C & -\Id_n \cr K & \0\end{pmatrix} -\lambda \begin{pmatrix} -M & \0 \cr \0 & -\Id_n \end{pmatrix}) \left(\begin{array}{c |||c}
	\Id_n & \0 \cr \hline\hline\hline
	\0 & Q_K
	\end{array}\right) \\
	&=&\left(\begin{array}{c |||c}
	Q^*_KC & -\Id_n \cr \hline\hline\hline
	{R}_{K,1} P^T_K & \0\cr\hline
	\0 & \0
	\end{array}\right) - \lambda \left(\begin{array}{c |||c}
	-Q^*_KM & \0 \cr \hline\hline\hline
	\0 & -\Id_n
	\end{array}\right). \label{eq:P1(..)Q1}
\end{eqnarray}
Note that $\mathrm{rank}(A) = n + \mathrm{rank}(K)$, so $A$ and $K$ have null spaces of equal dimensions. In essence, multiplication from the left with 
$Q_K^* \oplus Q_K^*$ (or with $Q_M^*\oplus Q_K^*$, or $\Id_n\oplus Q_K^*$) is a rank revealing transformation of $A$. 
We now truncate the $s_1=n-r_K$ copies of the eigenvalue $\lambda=0$ and proceed with the truncated $(n+r_K)\times (n+r_K)$ pencil
\begin{equation}\label{eq:truncatedPencil}
	A_{22} - \lambda B_{22} = \left(\begin{array}{c |||c}
	Q_{K,1}^*C & -\Id_{r_K} \cr \hline
	Q_{K,2}^*C & \0 \cr \hline\hline\hline
	{R}_{K,1}P^T_K & \0
	\end{array}\right) - \lambda \left(\begin{array}{c |||c}
	-Q^*_KM & \0 \cr \hline\hline\hline
	\0 & -\Id_{r_K}
	\end{array}\right).
\end{equation}
Note  that using the definition (\ref{eq:truncatedPencil}) of $A_{22} - \lambda B_{22}$ in {(\ref{eq:P1(..)Q1})} yields 
\begin{equation}
\mathbf{P}_1 ( A - \lambda B ) \mathbf{Q}_1 = \begin{pmatrix} A_{22}-\lambda B_{22} &  [\bigstar]_{n+r_K,n-r_K} \cr \0 & - \lambda B_{11}
\end{pmatrix},\;\; B_{11} = -\Id_{n-r_K}.
\end{equation}
At this point, \texttt{quadeig} removes the $n-r_K$ zero eigenvalues and leaves to the QZ to find the remaining ones, if present in the reduced problem. 
However, note that with $A_{11}:=A$ and $B_{11}:=B$, this procedure can be understood as the first step of the Van Dooren's algorithm (actually its transposed version described in \S \ref{SS=KCF-QR}; see (\ref{eq:rrQRKronecker2})), which removes the first Jordan block  of the eigenvalue zero. The existence of further Jordan blocks of $\lambda=0$ can be established immediately by checking the rank of $A_{22}$.

From the definition of $A_{22}$ in (\ref{eq:truncatedPencil}), it follows that  it suffices to  compute a rank revealing QR factorization of its $n\times n$ submatrix\footnote{This is why we have not followed the strategy (\ref{eq:LGEVP:n+rK}).} $A_{22}(r_K+1 : n+r_K, 1:n )$, 
	\begin{equation}\label{eq:intermediateMatrix}
	\left(\begin{array}{c}
	Q^*_{K,2}C \cr \hline
	{R}_{K,1}P^T_K
	\end{array}\right)P_{A_{22}} = Q_{A_{22}}R_{A_{22}}.
	\end{equation}
%
This can be used to transform the pencil $A_{22}-\lambda B_{22}$ to 
	\begin{equation}\label{eq:seconInter}
	\widehat{\mathbf{P}}_2(A_{22}-\lambda B_{22}) = \left(\begin{array}{c |||c}
	Q_{K,1}^*C & -\Id_{r_K} \cr \hline
	R_{A_{22}}P^T_{A_{22}} & \0
	\end{array}\right) - \lambda \widehat{\mathbf{P}}_2\left(\begin{array}{c |||c}
	Q^*_KM & \0 \cr \hline\hline\hline
	\0 & -\Id_{r_K}
	\end{array}\right),\;\; \widehat{\mathbf{P}}_2 = \left(\begin{array}{c| c}
	\Id_{r_K} & \0 \cr \hline
	\0 & Q^*_{A_{22}}
	\end{array}\right) .
	\end{equation}
If the factorization (\ref{eq:intermediateMatrix}) shows no rank deficiency,
there are no zeros in the spectrum of $A_{22}-\lambda B_{22}$. Equivalently, the reversed pencil $B_{22} - \mu A_{22}$ has no infinite eigenvalues.

Otherwise, $A_{22}(r_K+1 : n+r_K, 1:n )$ is rank deficient; assume its rank to be $r_{22}$, $r_{22} < n$, and let $s_2 = n+r_K-r_{22}$. Then
$
R_{A_{22}} = \left(\begin{smallmatrix}  {R}_{A_{22},1} \cr \0_{n-r_{22},n}
 \end{smallmatrix}\right),\;\;{R}_{A_{22},1} \in\mathbb{C}^{r_{22}\times n},
$
and
\begin{equation}\label{newDelationStep2}
\widehat{\mathbf{P}}_2 A_{22} = \left(\begin{array}{c |||c}
Q_{K,1}^*C & -\Id_{r_K} \cr \hline
{R}_{A_{22},1}P_{A_{22}}^T & \0 \cr \hline
\0_{n-r_{22},n} & \0_{n-r_{22},r_K}
\end{array}\right),\;\;
\widehat{\mathbf{P}}_2 B_{22} =
\left(\begin{array}{c |||c}
Q_{K,1}^*M & \0_{r_K} \cr \hline
\blacksquare_{r_{22},n} & \blacktriangle_{r_{22},r_K} \cr \hline
\square_{n-r_{22},n} & \triangle_{n-r_{22},r_K}
\end{array}\right) .
\end{equation}
%

To deflate the additional $m-r_{22}$ zeros, we need $(n-r_{22})\times (r_K+r_{22})$ zero block in the lower left corner of the matrix $\widehat{\mathbf{P}}_2 B_{22}$. 
This is achieved using the complete orthogonal decomposition $\left(\begin{array}{c|||c}
	\square & \triangle
\end{array}\right) = U_BR_BV^*_B$ of the $(n-r_{22})\times (n+r_k)$ block $\left(\begin{array}{c|||c}
\square & \triangle
\end{array}\right)$ in (\ref{newDelationStep2}).

If we write 
$
\left(\begin{array}{c|||c}
\square & \triangle
\end{array}\right)V_B = \left(\begin{array}{c c}
\Diamond_{n-r_{22},n-r_{22}} & \0_{n-r_{22}, r_K+r_{22}}
\end{array}\right)
$, $U_B R_B = \Diamond_{n-r_{22},n-r_{22}}$, 
and define the permutation $P_B$ to swap the 
first $(n-r_{22})$ and the remaining $(r_K+r_{22})$ columns, then deploying the transformation matrix 
\begin{equation}
\widehat{\mathbf{Q}}_2 = \left(\begin{array}{c c}
V_BP_B & \0_{n+r_K,n-r_K}\\
\0_{n-r_K,n+r_K} & \Id_{n-r_K}
\end{array}\right),
\end{equation}
from the right yields the targeted structure
\begin{equation}\label{eq:SecondTransformation}
\widehat{\mathbf{P}}_2A_{22}\widehat{\mathbf{Q}}_2-\lambda\widehat{\mathbf{P}}_2B_{22}\widehat{\mathbf{Q}}_2 = \left(\begin{array}{c c}
A_{33} - \lambda B_{33} & \spadesuit \\
\0_{n-r_{22},r_K+r_{22}} & -\lambda B_{22}
\end{array}\right).
\end{equation}
In \eqref{eq:SecondTransformation}, we slightly abuse the notation and denote the block $U_B R_B$ by $B_{22}$, to emphasize that this procedure follows the reduction from \S \ref{SS=KCF-QR}. If we premultiply \eqref{eq:SecondTransformation} with $\Id_{t_K+r_{22}}\oplus U_B^*$, then we can set $B_{22}=R_B$. Note that in (\ref{eq:SecondTransformation}) nonsingularity of $B_{22}$ (i.e. full row rank of the matrix block $\left(\begin{array}{c|||c}
\square & \triangle
\end{array}\right)$ in (\ref{newDelationStep2})) is necessary for the regularity of the pencil. If the overall algorithm is designed to consider singular pencil an invalid input, then instead of the complete orthogonal decomposition we can use rank revealing LQ factorization with row (or complete) pivoting; in that case $U_B$ is a permutation matrix, and $R_B$ is lower triangular.\footnote{In this case, by additional permutations $R_B$ can be made upper triangular.}

After this step, the structure of the linearization pencil is lost. To determine the existence of more zero eigenvalues, the rank of the $(r_K+r_{22})\times(r_K+r_{22})$ matrix $A_{33}$ is needed. For the further deflation we simply switch to the KCF computation as in \S \ref{SS=KCF-QR}.

\subsection{Backward error}\label{SS=KVADeig-backward}

	Since all involved transformations are numerically unitary, the classical backward error analysis of the transformations of $A-\lambda B$  in a unitary invariant matrix norm is straightforward. It is important, however, to keep the backward errors in terms of the original matrices $M$, $C$, $K$ (instead of the matrices $A$, $B$ of the linearization) and to estimate them in finer resolution -- instead of measuring the backward error in terms of matrix (operator) norm, the goal is to obtain the size of relative error in individual columns, rows or matrix entries.

   Such an analysis is usually technically involved and tedious. Here we provide the key elements of the analysis of the first two steps of the removal of zero eigenvalues. The idea is to provide insight, and to show how to keep small backward errors in the matrices of the quadratic problem, and to ensure it is column--wise small. 
   \begin{theorem}
   	The reduction described in \S \ref{SS=New-Reduce-KCF} computes (\ref{eq:SecondTransformation}) with small backward errors in the initial data matrices $M$, $C$ and $K$.
   	\end{theorem}
   
   We provide the proof with detailed discussion in the following two subsections, which  also contain the bounds on the backward errors.
\subsubsection{\underline{The first step}}	
	The backward error of the rank revealing factorization of $K$ is as explained in \S \ref{SS=RRQR-Back-err}; it contains both the floating point error $\delta K$ and the truncation error $\Delta K$ analogous to (\ref{eq:M:truncate:backw}), i.e. $(K+\delta K + \Delta K)\widetilde{P}_K = \widehat{Q}_K\widetilde{R}_K = \widehat{Q}_{K,1}\widetilde{R}_{K,1}$.
	Set $\Delta_{\Sigma}K=\delta K + \Delta K$, and introduce block partitions $\widehat{Q}_K = (\widehat{Q}_{K,1}, \widehat{Q}_{K,2})$, where $\widehat{Q}_{K,1}$ has $\widetilde{r}_K$ columns, and $\widetilde{r}_K$ is the computed numerical rank of $K$.	
	
	Let $\widetilde{X}_{11} = computed(\widetilde{Q}_K^* C)$.
	By the standard error analysis of floating point matrix multiplication we know hat 
	$
	computed(\widetilde{Q}_K^* C) = \widetilde{Q}_K^* C + \mathfrak{G}_C,\;\;
	|\mathfrak{G}_C| \leq \epsilon_{*}|\widetilde{Q}_K^*| |C|$, where $0\leq\epsilon_* \leq O(n)\roff$.
	Here we assume that $\widetilde{Q}_K$ is explicitly computed\footnote{As in LAPACK's subroutines \texttt{xORGQR}.} and then applied. If $\widetilde{Q}_K$ is stored and applied in factored form (using Householder vectors),\footnote{See e.g. LAPACK's subroutines \texttt{xORMQR}.} similar relation holds, but with more technical details, which we omit for the sake of brevity. 
	Since we can represent the numerically unitary $\widetilde{Q}_K$ as  $\widetilde{Q}_K = (\Id + \mathfrak{E}_C)\widehat{Q}_K$, $\|\mathfrak{E}_C\|_2\leq \epsilon_{qr}$, 
	$$
	computed(\widetilde{Q}_K^* C) = \widehat{Q}_K^* (\Id +\mathfrak{E}_C^*) C + \mathfrak{G}_C = \widehat{Q}_K^* ( C + \mathfrak{E}_C^* C + \widehat{Q}_K\mathfrak{G}_C) \equiv \widehat{Q}_K^* ( C + \delta C) ,
	$$
	with column-wise estimates $\|\delta C(:,i)\|_2 \leq ( \|\mathfrak{E}_C^*\|_2+\epsilon_* \sqrt{n}(1+\|\mathfrak{E}_C^*\|_2)) \| C(:,i)\|_2$.
	
	By the same reasoning we get $\widetilde{Y}_{11}\equiv computed(\widetilde{Q}_K^* M) = \widehat{Q}_K(M+\delta M)$, where for each $i=1,\ldots, n$, $\|\delta M(:,i)\|_2 \leq \epsilon_M \|M(:,i)\|_2$, and $\epsilon_M = (\epsilon_{qr}+\epsilon_*\sqrt{n}(1+\epsilon_{qr})).$
	
 Hence, we can represent the computed equivalent of (\ref{eq:P1(..)Q1}) as  
%
\vspace{-2mm}
	\begin{eqnarray*}
		\!\!\!&&\!\!\! \left(\begin{smallmatrix} \widehat{Q}_K^* & \0 \cr \0 & \widehat{Q}_K^* \end{smallmatrix}\right)\! \overbrace{\left\{ \begin{pmatrix} C+\delta C & -\Id_n \cr K+\Delta_{\Sigma} K & \0\end{pmatrix} -\lambda \begin{pmatrix} -M-\delta M & \0 \cr \0 & -\Id_n \end{pmatrix}\right\}\!}^{A+\delta A - \lambda(B+\delta B)} \left(\begin{smallmatrix} \Id_n & \0 \cr \0 & \widehat{Q}_K \end{smallmatrix}\right) 
		 \\
		\!\!\!&=& \!\!
		\left( \begin{array}{c|||c}  \widetilde{X}_{11} &   -\Id_{n}  \cr\hline\hline\hline \begin{array}{c} \widetilde{R}_{K,1}\widetilde{P}^T_K \cr\hline \0_{n-\widetilde{r}_K,n} \end{array} & 
			\begin{array}{c|c} \0_{\widetilde{r}_K,\widetilde{r}_K} & \0_{\widetilde{r}_K,n-\widetilde{r}_K}\cr\hline \0_{n-\widetilde{r}_K,\widetilde{r}_K} & \0 \end{array}\end{array} \right) 
		-\lambda \left(\begin{array}{c|||c} -\widetilde{Y}_{11} & \0 \cr\hline\hline\hline \0 & \begin{array}{c|c} -\Id_{\widetilde{r}_K} & \0 \cr\hline \0 &-\Id_{n-\widetilde{r}_K}\end{array}\end{array}\right) .
	\end{eqnarray*}
	
\subsubsection{\underline{The second step}}
In the next step, we compute the rank revealing factorization (\ref{eq:intermediateMatrix}) of the computed matrix, or, equivalently, 
\begin{equation}
	\begin{pmatrix}
	\widehat{Q}_{K,2}^*(C+\delta C)\cr
	\widetilde{R}_{K,1}\widetilde{P}_K^T
	\end{pmatrix}\Pi_A = Q_{A_{22}}R_{A_{22}}.\;\;(\mbox{Note that } \left(\begin{smallmatrix}
		\widehat{Q}_{K,2}^*(C+\delta C)\cr
		\widetilde{R}_{K,1}\widetilde{P}_K^T
		\end{smallmatrix}\right) = \left(\begin{smallmatrix}
		\widetilde{X}_{11}(\widetilde{r}_K+1:n,1:n)\cr
		\widetilde{R}_{K,1}\widetilde{P}_K^T
		\end{smallmatrix}\right).)
\end{equation}
For the computed factors $\widetilde{\Pi}_A, \widetilde{Q}_{A_{22}}, \widetilde{R}_{A_{22}}$ of floating point implementation it holds that
\begin{equation}\label{eq:backerr-push-back}
	\left[ \begin{pmatrix}
	\widehat{Q}_{K,2}^*(C+\delta C)\cr
	\widetilde{R}_{K,1}\widetilde{P}_K^T
	\end{pmatrix} + \begin{pmatrix}
	\mathfrak{C}\\
	\mathfrak{K}
	\end{pmatrix} \right] \widetilde{\Pi}_A = \widehat{Q}_{A_{22}}\widetilde{R}_{A_{22}} \equiv \begin{pmatrix}
	\widehat{Q}_{K,2}^*(C+\delta C + \widehat{Q}_{K,2}\mathfrak{C})\cr
	\widehat{Q}_{K,1}^*(K + \Delta_{\Sigma}K + \widehat{Q}_{K,1} \mathfrak{K})
	\end{pmatrix},
\end{equation}	 
where $\widehat{Q}_{A_{22}}$ is exactly unitary and close to the computed unitary factor $\widetilde{Q}_{A_{22}}$, and $\left(\begin{smallmatrix}
\mathfrak{C}\\
\mathfrak{K}
\end{smallmatrix}\right)$ is the backward error.
The last relation in (\ref{eq:backerr-push-back}) illustrates how to push the backward errors $\mathfrak{C}$, $\mathfrak{K}$ into the original data.

\begin{eqnarray}
&& \left(\begin{smallmatrix} \Id_{\widetilde{r}_K} & \0 & \0 \cr \0 & \widehat{Q}_{A_{22}}^* & \0 \cr \0 & \0 & \Id_{n-\widetilde{r}_K} \end{smallmatrix}\right)
\left(\begin{smallmatrix} \widehat{Q}_K^* & \0 \cr \0 & \widehat{Q}_K^* \end{smallmatrix}\right)\! \begin{pmatrix} C+\delta C +\widehat{Q}_{K,2}\mathfrak{C} & -\Id_n \cr K+\Delta_{\Sigma} K +\widehat{Q}_{K,1} \mathfrak{K} & \0\end{pmatrix} \left(\begin{smallmatrix} \Id_n & \0 \cr \0 & \widehat{Q}_K \end{smallmatrix}\right) \label{eq:A-part-backw}\\
&=& \left( \begin{array}{c|c}  \widetilde{X}_{11}(1:\widetilde{r}_K,1:n) &   \begin{array}{c||c} -\Id_{\widetilde{r}_K} & \0_{\widetilde{r}_K,n-\widetilde{r}_K} \end{array}  \cr\hline \begin{array}{c} \widetilde{R}_{A_{22},1}\widetilde{P}^T_{A_{22}} \cr\hline \0 \cr \hline\hline\0_{n-\widetilde{r}_K,n} \end{array} & 
\begin{array}{c||c} \0_{\widetilde{r}_{22},\widetilde{r}_K} & \widehat{Q}_{A_{22}}^*(1:\widetilde{r}_{22},1:n-\widetilde{r}_K) \cr\hline
\0_{n-\widetilde{r}_{22},\widetilde{r}_K} & \widehat{Q}_{A_{22}}^*(\widetilde{r}_{22}+1:n,1:n-\widetilde{r}_K) \cr\hline\hline \0_{n-\widetilde{r}_K,\widetilde{r}_K} & \0 \end{array}\end{array} \right) \\
&\approx&
\left( \begin{array}{c|c}  \widetilde{X}_{11}(1:\widetilde{r}_K,1:n) &   \begin{array}{c||c} -\Id_{\widetilde{r}_K} & \0_{\widetilde{r}_K,n-\widetilde{r}_K} \end{array}  \cr\hline \begin{array}{c} \widetilde{R}_{A_{22},1}\widetilde{P}^T_{A_{22}} \cr\hline \0 \cr \hline\hline\0_{n-\widetilde{r}_K,n} \end{array} & 
\begin{array}{c||c} \0_{\widetilde{r}_{22},\widetilde{r}_K} & \widetilde{Q}_{A_{22}}^*(1:\widetilde{r}_{22},1:n-\widetilde{r}_K) \cr\hline
\0_{n-\widetilde{r}_{22},\widetilde{r}_K} & \widetilde{Q}_{A_{22}}^*(\widetilde{r}_{22}+1:n,1:n-\widetilde{r}_K) \cr\hline\hline \0_{n-\widetilde{r}_K,\widetilde{r}_K} & \0 \end{array}\end{array} \right). \label{eq:A-part-forward-mixed} 
\end{eqnarray} 
This allows a mixed stability interpretation indicated by $"\approx"$ above: if the computed part of the matrix that contains $\widetilde{Q}_{A_{22}}$ is changed by using the corresponding submatrices of $\widehat{Q}_{A_{22}}$ instead, then we have backward stability.\footnote{This mixed formulation is preferred because it allows leaving the identity block unchanged, which leaves the underlying structure of the quadratic pencil intact. Besides,  $\widehat{Q}_{A_{22}}$ and $\widetilde{Q}_{A_{22}}$ are $O(\varepsilon_{qr})$ close.} Further, the backward stability holds for the reduced matrix (the leading $(n+\widetilde{r}_K)\times (n+\widetilde{r}_K)$) block.

Consider now the $B$--part of the pencil. The effective change is
\begin{eqnarray}
\!\! && computed(\widetilde{Q}_{A_{22}}^* \begin{pmatrix} -\widetilde{Y}_{11}(\widetilde{r}_K+1:n,1:n) & \0 \cr \0 & -\Id_{\widetilde{r}_K}\end{pmatrix}) \equiv \widetilde{Y}_{21}\equiv 
\widehat{Y}_{21} + \delta\widehat{Y}_{21}\\ \!\!&=&\!\!
 \widehat{Q}_{A_{22}}^* \!\!\begin{pmatrix} -\widehat{Q}_{K,2}^* (M+\delta M + \widehat{Q}_{K,2}\mathfrak{E}) & \0 \cr \0 & -\Id_{\widetilde{r}_K}\end{pmatrix} \!\! +\!\! \begin{pmatrix} \0 & \delta\mathfrak{Q}(1:n-\widetilde{r}_K,n-\widetilde{r}_K+1:n) \cr \0 & \delta\mathfrak{Q}(n-\widetilde{r}_K+1:n,n-\widetilde{r}_K+1:n) \end{pmatrix}
\nonumber 
\end{eqnarray}
where $\delta\mathfrak Q = \widetilde{Q}_{A_{22}}^* \!\! -\! \widehat{Q}_{A_{22}}^*$.
Note that $\widetilde{Y}_{11}(1:\widetilde{r}_K,1:n)=\widehat{Q}_{K,1}^* (M+\delta M) = \widehat{Q}_{K,1}^* (M+\delta M + \widehat{Q}_{K,2}\mathfrak{E})$. Hence, we can push $\widehat{Q}_{K,2}\mathfrak{E}$ backward into $M$ without affecting the first $\widetilde{r}_K$ rows in the $B$--part. 
Here too we have a mixed stability interpretation
\begin{eqnarray}
\!\!\!&& \left(\begin{smallmatrix} \Id_{\widetilde{r}_K} & \0 & \0 \cr \0 & \widehat{Q}_{A_{22}}^* & \0 \cr \0 & \0 & \Id_{n-\widetilde{r}_K} \end{smallmatrix}\right)
\left(\begin{smallmatrix} \widehat{Q}_K^* & \0 \cr \0 & \widehat{Q}_K^* \end{smallmatrix}\right)\! 
\begin{pmatrix} -(M+\delta M + \widehat{Q}_{K,2}\mathfrak{E}) & \0 \cr \0 & -\Id_n \end{pmatrix}
 \left(\begin{smallmatrix} \Id_n & \0 \cr \0 & \widehat{Q}_K \end{smallmatrix}\right) \label{eq:B-part-backw} \\
\!\!\! & = & \left(\begin{smallmatrix} \Id_{\widetilde{r}_K} & \0 & \0 \cr \0 & \widehat{Q}_{A_{22}}^* & \0 \cr \0 & \0 & \Id_{n-\widetilde{r}_K} \end{smallmatrix}\right)  \left(\begin{array}{c|||c}  \begin{array}{l}-\widetilde{Y}_{11}(1:\widetilde{r}_K,1:n) \cr\hline -\widehat{Q}_{K,2}^* (M+\delta M + \widehat{Q}_{K,2}\mathfrak{E}) \end{array} & \0 \cr\hline\hline\hline \0 & \begin{array}{c|c} -\Id_{\widetilde{r}_K} & \0 \cr\hline \0 &-\Id_{n-\widetilde{r}_K}\end{array}\end{array}\right) \cr
\!\!\! & = &\left(\begin{array}{c|||c}  \begin{array}{c}-\widetilde{Y}_{11}(1:\widetilde{r}_K,1:n) \cr\hline \widetilde{Y}_{21}(1:n-\widetilde{r}_K,1:n) \end{array} & \begin{array}{c|c} \0 & \0 \cr\hline \widetilde{Y}_{21}(1:n-\widetilde{r}_K,n+1:n+\widetilde{r}_K) & \0 \end{array} \cr\hline\hline\hline \begin{array}{c} \widetilde{Y}_{21}(n-\widetilde{r}_K+1:n,1:n) \cr\hline \0 \end{array} & \begin{array}{c|c} \widetilde{Y}_{21}(n-\widetilde{r}_K+1:n,n+1:n+\widetilde{r}_K) & \0 \cr\hline \0 &-\Id_{n-\widetilde{r}_K}\end{array}\end{array}\right) \nonumber \\
&+& 
\left(\begin{array}{c|||c}  \begin{array}{c}\0 \cr\hline \0 \end{array} & \begin{array}{c|c} \0 & \0 \cr\hline \delta\widehat{Y}_{21}(1:n-\widetilde{r}_K,n+1:n+\widetilde{r}_K) & \0 \end{array} \cr\hline\hline\hline \begin{array}{c} \0 \cr\hline \0 \end{array} & \begin{array}{c|c} \delta\widehat{Y}_{21}(n-\widetilde{r}_K+1:n,n+1:n+\widetilde{r}_K) & \0 \cr\hline \0 &\0\end{array}\end{array}\right) \label{eq:B-part-forward-mixed}
\end{eqnarray}
Hence, up to an $O(\varepsilon_{qr})$ forward error (\ref{eq:B-part-forward-mixed}), the computed $B$--part satisfies the backward perturbation relation (\ref{eq:B-part-backw}) with exactly unitary transformations. This, together with (\ref{eq:A-part-backw}), (\ref{eq:A-part-forward-mixed}) establishes a mixed stability interpretation of the second reduction step.

Mixed stability can be replaced by pure backward stability as follows. In (\ref{eq:backerr-push-back}), write $\widetilde{Q}_{A_{22}}\widetilde{R}_{A_{22}}$ instead of $\widehat{Q}_{A_{22}}\widetilde{R}_{A_{22}}$; this only changes the backward error $\left( \begin{smallmatrix} \mathfrak{C}\cr \mathfrak{K}\end{smallmatrix}\right)$. Then repeat (\ref{eq:A-part-backw})--(\ref{eq:A-part-forward-mixed}) and (\ref{eq:B-part-backw})--(\ref{eq:B-part-forward-mixed}) (with $\widehat{Q}_{A_{22}}\widetilde{R}_{A_{22}}$ replaced by $\widetilde{Q}_{A_{22}}\widetilde{R}_{A_{22}}$) with updated $\mathfrak{E}$ and $\delta\widehat{Y}_{21}=\0$.

\subsubsection{Comments on the sizes of the backward errors $\mathfrak{C}$ and $\mathfrak{K}$}
The backward error of the rank revealing QR factorization (\ref{eq:backerr-push-back}) is column-wise bounded by 
\begin{equation}
\left\| \begin{pmatrix}
\mathfrak{C}\\
\mathfrak{K}
\end{pmatrix} (:,i) \right\|_2 \leq \epsilon_{qr} \left\| \begin{pmatrix}
\widehat{Q}_{K,2}^*(C+\delta C)\\ \hline
\widetilde{R}_K\widetilde{P}_K^T
\end{pmatrix} (:,i) \right\|_2, \;\;i=1,\ldots, n.
\end{equation}
Here, for simplicity, we assume full rank case, so the backward error above does not contain the truncation part. The backward errors in $C$ and $K$ are then bounded by
\begin{displaymath}
\max\{ \| \mathfrak{C}(:,i)\|_2, \| \mathfrak{K}(:,i)\|_2 \}  \leq \epsilon_{qr} \sqrt{2} \max \left( \|\widehat{Q}_{K,2}^*(C+\delta C)(:,i)\|_2, \|\widetilde{R}_K\widetilde{P}_K^T(:,i)\|_2  \right).
\end{displaymath}
Since both $\| \mathfrak{C}(:,i)\|_2$, $\| \mathfrak{K}(:,i)\|_2$ contain mixing of $C$ and $K$, we see the benefits of scaling and balancing which then help keeping the error in each matrix small relative to its corresponding part/column.

\section{Putting it all together: the global structure of \texttt{KVADeig}}\label{S=KVADEIG-global}
Here, we describe the global structure of the new procedure. We assume that the initial scaling and balancing are done as requested by an expert user. 
\subsection{Full reduction scheme}\label{SS=FULL-REDUCTION}

We will always attempt to deflate zero eigenvalues. This means that if there are infinite eigenvalues, we will switch to the reversed problem. If both zero and infinite eigenvalues are present, we  work with the problem in which the number of zero eigenvalues is higher, and the structure is used in the deflation process of this eigenvalue. However, the information about the known number of infinite eigenvalues is transmitted in the next steps of the Algorithm \ref{a:Deflation0}.
The first step are the rank revealing decompositions of the matrices $M$ and $K$. Then we have the following cases.
\paragraph{1. Both matrices $M$ and $K$ are regular} In this case, there is no deflation. However, the rank revealing decomposition of $M$ is used to reduce the matrix $B$ to upper triangular form, as in \texttt{quadeig}, to facilitate more efficient run of the QZ algorithm.
\begin{equation*}
 \begin{pmatrix} Q_M^* & \0 \cr \0 & \Id_n \end{pmatrix} \!\!\left\{ \!\left(\begin{smallmatrix} C & -\Id_n \cr K & \0\end{smallmatrix}\right)\!\! -\!\!{\lambda} \left(\begin{smallmatrix} -M & \0 \cr \0 & -\Id_n \end{smallmatrix}\right)\!\right\}\!\!\begin{pmatrix} \Pi_M & \0 \cr \0 & \Id_n \end{pmatrix}\!\! = \!\!
\left(\begin{array}{c|||c} Q_M^* C\Pi_M & -Q_M^* \cr\hline\hline\hline K\Pi_M & \0 \end{array}\right) \!\!-\!\!\lambda \left(\begin{array}{c|||c} -R_M & \0 \cr\hline\hline\hline \0 & -\Id_n\end{array}\right)\!\!. \label{eq:LGEVP:MK-reg}
\end{equation*}

\paragraph{2. One of the matrices $M$ or $K$ is singular} Assume that $K$ is singular. Otherwise, switch to the reversed problem.
Before we start the deflation process of the $n-\rank(K) = n-r_K$ zero eigenvalues, we check whether there is more than one Jordan block for this eigenvalue, that is, we compute the numerical rank of the $n\times n$ block matrix $\left(\begin{smallmatrix}
Q^*_{K,2}C \cr \hline
{R}_{K,1}P^T_K
\end{smallmatrix}\right)$, see (\ref{eq:intermediateMatrix}). Depending on the computed rank, we have two cases.
\paragraph{2.1. Full rank} There are exactly $n\! -\! r_K$ zero eigenvalues. The deflation is done as in \texttt{quadeig}, i.e. along with the deflation, $B$ is reduced to upper triangular form.
\begin{align}
& \begin{pmatrix} Q_M^* & \0 \cr \0 & Q_K^* \end{pmatrix}\!\! \left\{\!\! \left(\begin{smallmatrix} C & -\Id_n \cr K & \0\end{smallmatrix}\right)\!\! -\!\!\lambda \left(\begin{smallmatrix} -M & \0 \cr \0 & -\Id_n \end{smallmatrix}\right)\!\!\right\}\!\!\begin{pmatrix} \Pi_M & \0 \cr \0 & Q_K \end{pmatrix}\!\!\! = \!\!\!
{\footnotesize\left(\begin{array}{c|||c} Q_M^* C\Pi_M & -Q_M^* Q_K \cr\hline\hline\hline 
\begin{array}{c} \widehat{R}_K\Pi_K^T\Pi_M \cr\hline \0_{n-r_K,n} \end{array} & \0 \end{array}\right) \!\!-\!\!\lambda \left(\begin{array}{c|||c} -R_M & \0 \cr\hline\hline\hline \0 & -\Id_n\end{array}\right)}  \nonumber\\
&\equiv 
\left( \begin{array}{c|||c}  X_{11} &  \begin{array}{c|c} X_{12} & X_{13} \end{array} \cr\hline\hline\hline \begin{array}{c} X_{21} \cr\hline \0_{n-r_K,n} \end{array} & 
\begin{array}{c|c} \0_{r_K,r_K} & \0_{r_K,n-r_K}\cr\hline \0_{n-r_K,r_K} & \0 \end{array}\end{array} \right) 
-\lambda \left(\begin{array}{c|||c} -R_M & \0 \cr\hline\hline\hline \0 & \begin{array}{c|c} -\Id_{r_K} & \0 \cr\hline \0 &-\Id_{n-r_K}\end{array}\end{array}\right)\label{eq:LGEVP:M-reg} .
\end{align}
The reduced $(n+r_K)\times (n+r_K)$ pencil is
\begin{equation}\label{eq:LGEVP:n+rK}
A -\lambda B = \left( \begin{array}{c|||c} 
X_{11} & X_{12} \cr\hline\hline\hline
X_{21} & \0_{r_K,r_K} \end{array}\right) - \lambda \left( \begin{array}{c|||c}
-R_M & \0 \cr\hline\hline\hline \0 & -\Id_{r_K}\end{array}\right),\;\;\begin{array}{l}
X_{11} = Q_M^* C\Pi_M \cr X_{12}= -(Q_M^* Q_K)(:,1:r_K)\cr
X_{21} = \widehat{R}_K\Pi_K^T\Pi_M
\end{array} .
\end{equation}
\paragraph{2.2. Rank deficient} This means that there is more than one Jordan block for the zero eigenvalue, more precisely there are at least $n-\rank({{\left(\begin{smallmatrix}
Q^*_{K,2}C \cr \hline
{R}_{K,1}P^T_K
\end{smallmatrix}\right)}}) + n-r_K = 2n-r_{22}-r_K$ zero eigenvalues. The first two blocks are deflated as described in \S \ref{S=New-Deflation}. Let 
(\ref{eq:SecondTransformation})
be the transformed pencil, after the two steps of the deflation. Now, the rank of the matrix $A_{33}$ determines whether there are more zero eigenvalues. The full deflation of zero eigenvalues is finished by calling the Algorithm \ref{a:Deflation0} on the truncated pencil $A_{33}-\lambda B_{33}$. As an output of the algorithm, we get the reduced pencil $A_{\ell+1,\ell+1}-\lambda B_{\ell+1,\ell+1}$ with $A_{\ell+1,\ell+1}$ regular, and the corresponding transformation matrices $Q_p$ and $P_p$ of order $r_K+r_{22}$. 
The final transformation matrices $Q$ and $P$ are
\begin{align}
	Q &= \left(\begin{array}{c |||c}
	\Id_n & \0 \cr \hline\hline\hline
	\0 & Q_K
	\end{array}\right) \left(\begin{array}{c c}
	V^*_BP_B & \0_{n+r_K,n-r_K}\\
	\0_{n-r_K,n+r_K} & \Id_{n-r_K}
	\end{array}\right)\diag(Q_p, \Id_{2n-r_K-r_{22}}) \\
	P &= \diag(P_p,\Id_{2n-r_K-r_{22}}) \left(\begin{array}{c c c}
	\Id_{r_K} & \0_{r_K,n} & \0_{r_K,n-r_K}\\
	\0_{n,r_K} & Q^*_{A_{22}} & \0_{n,n-r_K} \\
	\0_{n-r_K,r_K} & \0_{n-r_K,n}& \Id_{n-r_K}
	\end{array} \right) \left(\begin{array}{c |||c}
	Q^*_K & \0 \cr \hline\hline\hline
	\0 & Q^*_K
	\end{array}\right).
\end{align}
Now, the remaining nonzero eigenvalues are computed using the QZ algorithm on the truncated pencil $A_{\ell+1,\ell+1}-\lambda B_{\ell+1,\ell+1}$.
\paragraph{3. Both matrices $M$ and $K$ are singular}
In this case,  we first check the ranks of both block matrices
$		\left(\begin{smallmatrix}
		Q^*_{K,2}C \cr \hline
		\widehat{R}_KP^T_K
		\end{smallmatrix}\right), \;\;\left(\begin{smallmatrix}
		Q^*_{M,2}C \cr \hline
		\widehat{R}_MP^T_M
		\end{smallmatrix}\right).
$
The ranks of these matrices determine whether there exist more than one Jordan block for the zero and the infinite eigenvalue, respectively. Depending on their regularity, we distinguish the following three cases.
\paragraph{3.1. Both matrices are regular} This means that there are exactly $n-r_M$ infinite, and $n-r_K$ zero eigenvalues. The deflation of these eigenvalues is done analogously to the \texttt{quadeig} algorithm. For the reader's convenience, we provide the details. Let
\begin{align}
& \begin{pmatrix} Q_M^* & \0 \cr \0 & Q_K^* \end{pmatrix} \!\!\left\{\!\! \left(\begin{smallmatrix} C & -\Id_n \cr K & \0\end{smallmatrix}\right)\!\! -\!\!\lambda \left(\begin{smallmatrix} -M & \0 \cr \0 & -\Id_n \end{smallmatrix}\right)\!\!\right\}\!\!\begin{pmatrix} \Id_n & \0 \cr \0 & Q_K \end{pmatrix}\!\!\! =\!\!\!
{\footnotesize\left( \begin{array}{c||| c} Q_M^* C & -Q_M^* Q_K \cr\hline\hline\hline
\widehat{R}_K\Pi_K^T & \0_{r_K,n} \cr \0_{n-r_k,n} & \0_{n-r_K,n}\end{array}\right) \!\!- \!\!\lambda\!\! \left(\!\! \begin{array}{c||| c} -\widehat{R}_M\Pi_M^T & \0_{r_M,n} \cr 
\0_{n-r_M,n} & \0_{n-r_M,n} \cr\hline\hline\hline \0_{n,n} & -\Id_n \end{array}\!\!\right)}\nonumber \\
\!\!\!  &\equiv \!\!\!
{\small  \left(\!\! \begin{array}{c|||c}  \begin{array}{c|c} X_{11} & X_{12}\cr\hline X_{21} & X_{22}\end{array} &  \begin{array}{c|c} X_{13} & X_{14}\cr\hline X_{23} & X_{24}\end{array}  \cr\hline\hline\hline \begin{array}{c|c} X_{31} & X_{32}  \cr\hline \0 & \0 \end{array} & 
	\begin{array}{c|c} \0 & \0\cr\hline \0 & \0 \end{array}\end{array}\!\! \right)} \!\!-\! \lambda\!
{\small \left(\!\! \begin{array}{c|||c}  \begin{array}{c|c} Y_{11} & Y_{12}\cr\hline \0_{n-r_M,r_M} & \0_{n-r_M,n-r_M}\end{array} &  \begin{array}{c|c} \0_{r_M,r_K} & \0_{r_M,n-r_K} \cr\hline \0_{n-r_M,r_K} & \0_{n-r_M,n-r_K}\end{array}  \cr\hline\hline\hline  \begin{array}{c|c} \0_{r_K,r_M} & \0_{r_K,n-r_M}  \cr\hline \0_{n-r_K,r_M} & \0_{n-r_K,n-r_M} \end{array} & 
	\begin{array}{c|c} -\Id_{r_K} & \0_{r_K,n-r_K}\cr\hline \0_{n-r_K,r_K} & -\Id_{n-r_K} \end{array}\end{array} \!\!\right)}  \label{eq:X-lY}
\end{align} 
Note the difference in the transformation from the right with \eqref{eq:LGEVP:M-reg}: instead of $\Pi_M$, we now have $\Id_n$, so that $Q_M^* M = R_M\Pi_M^T$ is not upper triangular. Preserving the triangular form in this moment does not seem important because it is likely that it will be destroyed in subsequent steps.
In the next step, the complete orthogonal decomposition (i.e. URV decomposition, using unitary matrices $Q_X$ and $Z_X$) is computed as
\begin{equation*}
Q_X^* \begin{pmatrix} X_{21} & X_{22} & X_{23}\end{pmatrix} Z_X^* \begin{pmatrix} \0 & \Id_{n-r_M}\cr \Id_{r_M+r_K} & \0\end{pmatrix} = \begin{pmatrix} \0_{n-r_M,r_M+r_K} & R_X\end{pmatrix} .
\end{equation*}  
Then (\ref{eq:X-lY}) can be further transformed as follows:
\begin{eqnarray}
\!\!\!&&\begin{pmatrix} \Id_{r_M} & \0 & \0 & \0  \cr\hline
\0 & \0 & \Id_{r_K} & \0 \cr
\0 & Q_X^* & \0 & \0 \cr\hline
\0 & \0 & \0 & \Id_{n-r_K}\end{pmatrix}
\left( \begin{array}{c|||c}  \begin{array}{c|c} X_{11} & X_{12}\cr\hline X_{21} & X_{22}\end{array} &  \begin{array}{c|c} X_{13} & X_{14}\cr\hline X_{23} & X_{24}\end{array}  \cr\hline\hline\hline \begin{array}{c|c} X_{31} & X_{32}  \cr\hline \0 & \0 \end{array} & 
\begin{array}{c|c} \0 & \0\cr\hline \0 & \0 \end{array}\end{array} \right)
\left(\begin{array}{c|c} Z_X^* \left(\begin{smallmatrix} \0 & \Id_{n-r_M}\cr \Id_{r_M+r_K} & \0\end{smallmatrix}\right) & \0 \cr\hline \0 & \Id_{n-r_K} \end{array}\right) \nonumber \\
&\!\!=\!\!& \left( \begin{array}{c||c}  \begin{array}{c|c} \widetilde{X}_{11} & \widehat{X}_{12}\cr\hline \widetilde{X}_{21} & \widehat{X}_{22}\end{array} &  \begin{array}{c|c} \widehat{X}_{13} & X_{14}\cr\hline \widehat{X}_{23} & \0\end{array}  \cr\hline\hline \begin{array}{c|c} \0 & \0  \cr\hline \0 & \0 \end{array} & 
\begin{array}{c|c} R_X & \widetilde{X}_{24}\cr\hline \0 & \0 \end{array}\end{array} \right)\!,\;
\mbox{where}\;
\begin{array}{l}
\left(\begin{smallmatrix}
\widetilde{X}_{11} & \widetilde{X}_{12} & \widetilde{X}_{13}  \cr
\widetilde{X}_{21} & \widetilde{X}_{22} & \widetilde{X}_{23} \end{smallmatrix}
\right)=
\left(\begin{smallmatrix}
{X}_{11} & {X}_{12} & {X}_{13}\cr  {X}_{31} & {X}_{32} & \0 \end{smallmatrix}\right)
Z_X^*\left(\begin{smallmatrix} \0 & \Id_{n-r_M}\cr \Id_{r_M+r_K} & \0\end{smallmatrix}\right) \cr
\begin{smallmatrix} \widetilde{X}_{24} = Q_X^* X_{24} \end{smallmatrix} .
\end{array}
\nonumber
\end{eqnarray} 
The $(1,1)$ diagonal block in the new partition ($=\!\parallel\!=$) is $(r_M+r_K)\times(r_M+r_K)$, and 
\begin{itemize}
	\item \framebox{$n-r_M=r_K$}: $\widehat{X}_{ij}=\widetilde{X}_{ij}$, $i=1,2$, $j=2,3$;
	\item \framebox{$n-r_M>r_K$}: $\widehat{X}_{12}=\widetilde{X}_{12}(:,1:r_K)$, 
	$\widehat{X}_{13}=( \widetilde{X}_{12}(:,r_K+1:n-r_M),\;\widetilde{X}_{13})$, 
	
	$\widehat{X}_{22}=\widetilde{X}_{22}(:,1:r_K)$, $\widehat{X}_{23}=( \widetilde{X}_{22}(:,r_K+1:n-r_M),\;\widetilde{X}_{23})$
	
	\item \framebox{$n-r_M < r_K$}: $\widehat{X}_{12} = ( \widetilde{X}_{12},\; \widetilde{X}_{13}(:,1:r_K+r_M-n))$, $\widehat{X}_{13}=\widetilde{X}_{13}(:,r_K+r_M-n+1:r_K)$
	
	$\widehat{X}_{22} = ( \widetilde{X}_{22},\; \widetilde{X}_{23}(:,1:r_K+r_M-n))$, $\widehat{X}_{23}=\widetilde{X}_{23}(:,r_K+r_M-n+1:r_K)$
\end{itemize}
On the right hand side, the transformation reads, analogously, 
\begin{eqnarray*}
&&\begin{pmatrix} \Id_{r_M} & \0 & \0 & \0  \cr\hline
\0 & \0 & \Id_{r_K} & \0 \cr
\0 & Q_X^* & \0 & \0 \cr\hline
0 & \0 & \0 & \Id_{n-r_K}\end{pmatrix}
Y
\left(\begin{array}{c|c} Z_X^* \begin{pmatrix} \0 & \Id_{n-r_M}\cr \Id_{r_M+r_K} & \0\end{pmatrix} & \0 \cr\hline \0 & \Id_{n-r_K} \end{array}\right) \\
&=& 
\left( \begin{array}{c||c}  
\begin{array}{c|c} \widetilde{Y}_{11} & \widehat{Y}_{12}\cr\hline \widetilde{Y}_{21}  & \widehat{Y}_{22} \end{array} &  \begin{array}{c|c} \widehat{Y}_{13} & \0_{r_M,n-r_K} \cr\hline \widehat{Y}_{23} & \0_{r_K,n-r_K}\end{array}  \cr\hline\hline \begin{array}{c|c} \0_{n-r_M,r_M} & \0_{n-r_M,n-r_M}  \cr\hline \0_{n-r_K,r_M} & \0_{n-r_K,n-r_M} \end{array} & 
\begin{array}{c|c} \0_{n-r_M,r_K} & \0_{n-r_M,n-r_K} \cr\hline \0 & -\Id_{n-r_K} \end{array}\end{array} \right) \\
&&\mbox{where}\;\; \begin{pmatrix}\widetilde{Y}_{11} & \widetilde{Y}_{12} & \widetilde{Y}_{13} \cr  \widetilde{Y}_{21} & \widetilde{Y}_{22} & \widetilde{Y}_{23} \end{pmatrix} = \begin{pmatrix} Y_{11} & Y_{12} & \0_{r_M,r_K} \cr \0_{r_k,r_M}  & \0_{r_K,n-r_M} & -\Id_{r_K}\end{pmatrix} Z_X^* \begin{pmatrix} \0 & \Id_{n-r_M}\cr \Id_{r_M+r_K} & \0\end{pmatrix}\;\mbox{and}
\end{eqnarray*}
\begin{itemize}
	\item \framebox{$n-r_M=r_K$}: $\widehat{Y}_{ij}=\widetilde{Y}_{ij}$, $i=1,2$, $j=2,3$;
	\item \framebox{$n-r_M>r_K$}: $\widehat{Y}_{12}=\widetilde{Y}_{12}(:,1:r_K)$, 
	$\widehat{Y}_{13}=( \widetilde{Y}_{12}(:,r_K+1:n-r_M),\;\widetilde{Y}_{13})$, 
	
	$\widehat{Y}_{22}=\widetilde{Y}_{22}(:,1:r_K)$, $\widehat{Y}_{23}=( \widetilde{Y}_{22}(:,r_K+1:n-r_M),\;\widetilde{Y}_{23})$
	
	\item \framebox{$n-r_M < r_K$}: $\widehat{Y}_{12} = ( \widetilde{Y}_{12},\; \widetilde{Y}_{13}(:,1:r_K+r_M-n))$, $\widehat{Y}_{13}=\widetilde{Y}_{13}(:,r_K+r_M-n+1:r_K)$
	
	$\widehat{Y}_{22} = ( \widetilde{Y}_{22},\; \widetilde{Y}_{23}(:,1:r_K+r_M-n))$, $\widehat{Y}_{23}=\widetilde{Y}_{23}(:,r_K+r_M-n+1:r_K)$
\end{itemize}
\begin{equation*}
\mbox{Hence, the equivalent pencil is }
\left( \begin{array}{c||c}  \begin{array}{c|c} \widetilde{X}_{11} & \widehat{X}_{12}\cr\hline \widetilde{X}_{21} & \widehat{X}_{22}\end{array} &  \begin{array}{c|c} \widehat{X}_{13} & X_{14}\cr\hline \widehat{X}_{23} & \0\end{array}  \cr\hline\hline \begin{array}{c|c} \0 & \0  \cr\hline \0 & \0 \end{array} & 
\begin{array}{c|c} R_X & \widetilde{X}_{24}\cr\hline \0 & \0 \end{array}\end{array} \right)
-\lambda 
\left( \begin{array}{c||c}  
\begin{array}{c|c} \widetilde{Y}_{11} & \widehat{Y}_{12}\cr\hline \widetilde{Y}_{21}  & \widehat{Y}_{22} \end{array} &  \begin{array}{c|c} \widehat{Y}_{13} & \0 \cr\hline \widehat{Y}_{23} & \0\end{array}  \cr\hline\hline \begin{array}{c|c} \0 & \0  \cr\hline \0 & \0 \end{array} & 
\begin{array}{c|c} \0 & \0 \cr\hline \0 & -\Id \end{array}\end{array} \right) ,
\end{equation*}
and it immediately reveals that the original quadratic pencil is singular if $\mathrm{det}(R_X)=0$. Otherwise, we first identify the $n-r_K$ zero eigenvalues and the $n-r_M$ infinite ones, and the remaining ones are computed from  
the linear generalized eigenvalue problem 
of the $(r_M+r_K)\times (r_M+r_K)$ pencil
\begin{equation}\label{eq:deflated:A-lB}
A - \lambda B \equiv \left(\begin{array}{c|c} \widetilde{X}_{11} & \widehat{X}_{12}\cr\hline \widetilde{X}_{21} & \widehat{X}_{22}\end{array}\right) - \lambda
\left( \begin{array}{c|c} \widetilde{Y}_{11} & \widehat{Y}_{12}\cr\hline \widetilde{Y}_{21}  & \widehat{Y}_{22} \end{array} \right)
\end{equation} 
\paragraph{3.2. Only one of the matrices  is singular}
In any case, we use the structure of the linearization to deflate two Jordan blocks of the eigenvalue zero, meaning that if $\scriptscriptstyle \left(\begin{smallmatrix}
Q^*_{M,2}C \cr \hline
\widehat{R}_MP^T_M
\end{smallmatrix}\right)$ is singular, i.e if there are at least two Jordan block for infinite eigenvalues, we switch to the reversed problem.
Thus, we have the following: there are at least $n-\rank(K) + n-\rank(\left(\begin{smallmatrix}
Q^*_{K,2}C \cr \hline
\widehat{R}_KP^T_K
\end{smallmatrix}\right)) = n-r_K+n-r_{22}$ zero eigenvalues and exactly $n-r_M$ infinite eigenvalues.
First, two know Jordan blocks of zero eigenvalue are deflated as described in section \S \ref{S=New-Deflation}. As a result we get the transformed pencil as in \eqref{eq:SecondTransformation}.

As in the case 2.2, the existence of the additional zero eigenvalues depends on the rank of the matrix $A_{33}$, and deflation of zeros is completed by calling the Algorithm \ref{a:Deflation0} on the pencil $A_{33}-\lambda B_{33}$. The output arguments are the transformation matrices $Q_p, P_p$ of order $r_K+r_{22}$ and pencil $A_{\ell+1,\ell+1}-\lambda B_{\ell+1,\ell+1}$ with $A_{\ell+1,\ell+1}$ being regular.

At this point, we have removed all zero eigenvalues. It remains to remove $n-r_M$ infinite ones. This is done by using Algorithm \ref{a:Deflation0} on the reversed pencil $B_{\ell+1,\ell+1}-\lambda A_{\ell+1,\ell+1}$. Since we already determined that there is just one Jordan block for infinite eigenvalue, by computing the rank of $\left(\begin{smallmatrix}
Q^*_{M,2}C \cr \hline
\widehat{R}_MP^T_M
\end{smallmatrix}\right)$, we send, as input data of the Algorithm \ref{a:Deflation0}, the information about $\rank(B_{\ell+1,\ell+1}) = \rank(\left(\begin{smallmatrix}
Q^*_{M,2}C \cr \hline
\widehat{R}_MP^T_M
\end{smallmatrix}\right))$ and the number of the deflation steps, which is one. Finally, as the output we get a pencil $A_{\ell+2,\ell+2}-\lambda B_{\ell+2,\ell+2}$ with both $A_{\ell+2,\ell+2}$ and $B_{\ell+2,\ell+2}$ regular, and the transformation matrices $Q_{P1}, P_{P1}$.

\noindent Let $n_{\ell+1}$ be the dimension of $A_{\ell+1,\ell+1}\!-\!\lambda B_{\ell+1,\ell+1}$.The final transformations $Q, P$ are
\begin{align*}
	Q &= \left(\begin{array}{c |||c}
	\Id_n & \0 \cr \hline\hline\hline
	\0 & Q_K
	\end{array}\right) \left(\begin{array}{c c}
	V^*_BP_B & \0_{n+r_K,n-r_K}\\
	\0_{n-r_K,n+r_K} & \Id_{n-r_K}
	\end{array}\right)\diag(Q_p, \Id_{2n-r_K-r_{22}})\diag(Q_{p1}, \Id_{2n-n_{\ell+1}})\\
	P &=\diag(P_{p1}, \Id_{2n-n_{\ell+1}}) \diag(P_p,\Id_{2n-r_K-r_{22}}) \left(\begin{array}{c c c}
	\Id_{r_K} & \0_{r_K,n} & \0_{r_K,n-r_K}\\
	\0_{n,r_K} & Q^*_{A_{22}} & \0_{n,n-r_K} \\
	\0_{n-r_K,r_K} & \0_{n-r_K,n}& \Id_{n-r_K}
	\end{array} \right) \left(\begin{array}{c |||c}
	Q^*_K & \0 \cr \hline\hline\hline
	\0 & Q^*_K
	\end{array}\right).
\end{align*}

The remaining finite nonzero eigenvalues are computed using the QZ algorithm for the truncated pencil $A_{\ell+2,\ell+2}-\lambda B_{\ell+2,\ell+2}$.

\paragraph{3.3. Both matrices are singular}
The decision whether we consider original or the reversed problem depends on the total number of the detected zero and infinite eigenvalues. The goal is to use the structure of the linearization pencil to deflate the zeros. Hence, we will consider the problem with the higher number of the zeros.

Therefore, at the beginning of the deflation process have the following setting: there are at least $n-\rank(K) + n-\rank(\left(\begin{smallmatrix}
Q^*_{K,2}C \cr \hline
\widehat{R}_KP^T_K
\end{smallmatrix}\right)) = n-r_K+n-r_{22}$ zeros and $n-\rank(M) + n-\rank(\left(\begin{smallmatrix}
Q^*_{M,2}C \cr \hline
\widehat{R}_MP^T_M
\end{smallmatrix}\right))$ infinite eigenvalues.
First step is to deflate all detected zeros using the procedure described in the \S \ref{S=New-Deflation}. The resulting pencil is \eqref{eq:SecondTransformation}.

Possible extra zeros are deflated using the Algorithm \ref{a:Deflation0} on the $A_{33}-\lambda B_{33}$. As the output we get the pencil $A_{\ell +1, \ell +1}-\lambda B_{\ell+1,\ell+1}$ with $A_{\ell +1, \ell +1}$ regular, and $Q_p$ and $P_p$ transformation matrices. Denote with $n_{\ell+1}$ the dimension of the resulted pencil.

At this point, all the zero eigenvalues are deflated. Now, we want to deflate infinite eigenvalues. We already know the dimension of first two Jordan blocks, namely $n-r_M$ and $n-\rank(\left(\begin{smallmatrix}
Q^*_{M,2}C \cr \hline
\widehat{R}_MP^T_M
\end{smallmatrix}\right))$. This information is forwarded to the Algorithm \ref{a:Deflation0} along with the reversed pencil $ B_{\ell+1,\ell+1}-\lambda A_{\ell +1, \ell +1}$. In this case, the algorithm will not compute the numerical rank in the first two steps, but it will use the forwarded information. The output is the pencil $A_{\ell +\ell_1, \ell +\ell_1}-\lambda B_{\ell+1,\ell+1}$, with both $A_{\ell +\ell_1, \ell +\ell_1}$ and $B_{\ell+1,\ell+1}$ regular, and the transformation matrices $Q_{p1}$,  $Q_{p_2}$. 
The final transformation matrices $Q$, $P$ are 
\begin{align*}
Q &= \left(\begin{array}{c |||c}
\Id_n & \0 \cr \hline\hline\hline
\0 & Q_K
\end{array}\right) \left(\begin{array}{c c}
V^*_BP_B & \0_{n+r_K,n-r_K}\\
\0_{n-r_K,n+r_K} & \Id_{n-r_K}
\end{array}\right)\diag(Q_p, \Id_{2n-r_K-r_{22}})\diag(Q_{p1}, \Id_{2n-n_{\ell+1}})\\
P &=\diag(P_{p1}, \Id_{2n-n_{\ell+1}}) \diag(P_p,\Id_{2n-r_K-r_{22}})\!\! \left(\begin{array}{c c c}
\Id_{r_K} & \0_{r_K,n} & \0_{r_K,n-r_K}\\
\0_{n,r_K} & Q^*_{A_{22}} & \0_{n,n-r_K} \\
\0_{n-r_K,r_K} & \0_{n-r_K,n}& \Id_{n-r_K}
\end{array} \right) \!\!\!\left(\begin{array}{c |||c}
Q^*_K & \0 \cr \hline\hline\hline
\0 & Q^*_K
\end{array}\right)\! .
\end{align*}
The remaining (finite nonzero) eigenvalues are computed using the QZ algorithm on $A_{\ell +\ell_1, \ell +\ell_1}-\lambda B_{\ell+1,\ell+1}$.

\subsection{Eigenvector recovery}
The process of extracting the eigenvectors of the quadratic eigenvalue problem from those of the linearization is as follows (see also \cite{Hammarling:QUADEIG}).  The right eigenvectors $z$ and the left eigenvectors $w$ of the linearization $C_2(\lambda)$  are of the form
\begin{equation}\label{RightLeftEigenvector}
z = \left(\begin{array}{c}
z_1\cr \hline \hline \hline
z_2
\end{array}\right) = 
\left(\begin{array}{c}
\lambda x \cr
\hline \hline \hline
-Kx
\end{array}\right)\;\mbox{for}\;  \lambda \neq 0; \;\;
\left(\begin{array}{c}
\lambda x\cr \hline \hline \hline
C x
\end{array}\right)\;\mbox{for}\;  \lambda = 0;
\;\;\; w = \left(\begin{array}{c}
w_1\\ \hline \hline \hline
w_2
\end{array}\right) = \left(\begin{array}{c}
\lambda y\\ \hline \hline \hline 
y
\end{array}\right), 
\end{equation}	
where $x,y$ are the right and the left eigenvector for the quadratic problem. From this relation we see that, when the matrix $K$ is nonsingular, we have two choices for the right eigenvector, namely $z_1$ and $K^{-1}z_2$. Otherwise, we only have one choice,\footnote{Up to a nonzero scaling parameter.} $z_1$. 

The eigenvectors of $C_2(\lambda)$ are obtained from the eigenvectors $\widetilde{z}$, $\widetilde{w}$ of the transformed pencil $P C_2(\lambda)Q$ as $z=Q\widetilde{z}$, $w=P^*\widetilde{w}$ where $P$, $Q$ contain all transformations from the reduction process described in \S \ref{SS=FULL-REDUCTION}. For the readers convenience, we provide the details of assembling the eigenvectors.

\subsubsection{The right eigenvectors}
\paragraph{Case 1}
 The matrix $K$ is nonsingular. We have two choices for $x$. Let $\widetilde{z}$ be the right eigenvector for the transformed linearization. The corresponding right eigenvector for $C_2(\lambda)$ is
 $z = \left(\begin{array}{c}
 z_1\cr \hline \hline \hline
 z_2
 \end{array}\right) = \begin{pmatrix}
 \Pi_M & \mathbf{0}\\
 \mathbf{0} & \mathbb{I}_n
 \end{pmatrix} \left(\begin{array}{c}
 \widetilde{z}_1\cr \hline\hline\hline
 \widetilde{z}_2
 \end{array}\right) = \left(\begin{array}{c}
 \Pi_M \widetilde{z}_1\cr \hline\hline\hline
 \widetilde{z}_2
 \end{array}\right)$.
 Hence, the two candidates for the eigenvector $x$ are $\Pi_M\widetilde{z}_1$ and $K^{-1}\widetilde{z}_2$. We choose the one with smaller backward error.
\paragraph{Case 2.1}
The matrix $K$ is singular, and $n-r_K$ zero eigenvalues are deflated. Eigenvectors corresponding to those eigenvalues span the nullspace of matrix $K$. The basis for this nullspace is computed as orthogonal complement of the range of $K^*$ using the QR decomposition of the upper triangular matrix $\widehat{R}^*_K$, 
$\widehat{R}^*_K = Q_{\widehat{R}^*_K}R_{\widehat{R}^*_K}$.
The wanted vector is determined by the last $n-r_K$ columns of the orthogonal matrix $Q_{\widehat{R}^*_K}$.\\
The remaining eigenvalues and eigenvectors {$\widetilde{z}=\left( \begin{smallmatrix}\widetilde{z}_1\cr \widetilde{z}_2 \end{smallmatrix}\right) \in \mathbb{C}^{n+r_K}$} are computed from the transformed linearization (\ref{eq:LGEVP:n+rK}). The corresponding eigenvector for $C_2(\lambda)$ is
\begin{equation*}
\left(\begin{array}{c}
z_1\cr \hline\hline\hline
z_2
\end{array}\right) = \begin{pmatrix}
\Pi_M & \mathbf{0}\\
\mathbf{0} & Q_K
\end{pmatrix}\left(\begin{array}{c}
\widetilde{z}_1\cr\hline\hline\hline
\widetilde{z}_2\cr
0
\end{array}\right) = \left(\begin{array}{c}
\Pi_M \widetilde{z}_1\cr \hline\hline\hline
Q_K \begin{pmatrix}
\widetilde{z}_2\\
0
\end{pmatrix}
\end{array}\right).\;\;\mbox{Hence, $x=\Pi_M \widetilde{z}_1$.}
\end{equation*}
\paragraph{Case 3.1}
Both matrices $M$ and $K$ are singular, and $n-r_M$ infinite and $n-r_K$ zero eigenvalues are deflated. The eigenvectors for zero eigenvalues are obtained as in the case (2.1), whilst the eigenvectors for the infinite eigenvalues form a basis for the nullspace of the matrix $M$. The basis is obtained as the orthogonal complement of the range of $M^*$, represented by the last $n-r_M$ columns of the  orthogonal matrix $Q_{\widehat{R}^*_M}$ in the QR factorization
$\widehat{R}^*_M = Q_{\widehat{R}^*_M}R_{\widehat{R}^*_M}$.
The remaining eigenvalues and eigenvectors $\widetilde{z}\!\in\! \mathbb{C}^{r_K+r_M}\!$ are obtained from the pencil (\ref{eq:deflated:A-lB}). The corresponding eigenvector of $C_2(\lambda)$ is
\begin{equation*}
\left(\begin{array}{c}
z_1\cr \hline\hline\hline
z_2
\end{array}\right) = \begin{pmatrix}
\mathbb{I}_n & \mathbf{0}\\
\mathbf{0} & Q_K
\end{pmatrix}\begin{pmatrix}
Z^*_X & \mathbf{0}\\
\mathbf{0} & \mathbb{I}_{n-r_K}
\end{pmatrix} \!\!\begin{pmatrix}
\mathbf{0} & \mathbb{I}_{n-r_K} & \mathbf{0}\\
\mathbb{I}_{r_K+r_M} & \mathbf{0} & \mathbf{0}\\
\mathbf{0} & \mathbf{0} & \mathbb{I}_{n-r_K}
\end{pmatrix}\!\!\left(\begin{array}{c}
\widetilde{z}\cr \hline\hline
\0\cr
\0
\end{array}\right) = \begin{pmatrix}
\mathbb{I}_n & \mathbf{0}\\
\mathbf{0} & Q_K
\end{pmatrix}\!\!\begin{pmatrix}
Z^*_X \begin{pmatrix}
\0\\
\widetilde{z}
\end{pmatrix}\\
\0
\end{pmatrix}\! .
\end{equation*}
The eigenvector $x$ is represented by the first $n\!\times\! 1$ block of the $(n\!+\! r_K)\!\times\! 1$ vector $Z^*_X \begin{pmatrix}
\0\\
\widetilde{z}
\end{pmatrix}$.
\paragraph{Cases 2.2., 3.2. and 3.3} When Algorithm \ref{a:Deflation0} is used for the removal of additional zero and/or infinite eigenvalues, the eigenvector recovery procedure {is as follows}. Let $n_{\ell+j}$ be the dimension of the truncated pencil $A_{\ell+j,\ell+j}-\lambda B_{\ell+j,\ell+j}$, {$j=1,2,\ell_1$}. Let $\widetilde{z} \in \mathbb{R}^{n_{\ell+j}}$ be the computed right eigenvector  of $A_{\ell+j,\ell+j}-\lambda B_{\ell+j,\ell+j}$. If $n_{\ell+j}>n$, the right eigenvector is recovered as $x = Q(1:n,1:n)\widetilde{z}(1:n)$, and if $n_{\ell+j}<n$ then $x = Q(1:n,1:n_{\ell+\ell_1})\widetilde{z}$.

\subsubsection{The left eigenvectors}
\paragraph{Case 1}
Let $\widetilde{w}$ be the left eigenvector for the transformed linearization $PC_2(\lambda)Q$. The corresponding left eigenvector for the linearization $C_2(\lambda)$ is 
\begin{equation}
w = \left(\begin{array}{c}
w_1\cr \hline\hline\hline
w_2
\end{array}\right) = \begin{pmatrix}
Q_M & \mathbf{0}\\
\mathbf{0} & \mathbb{I}_n
\end{pmatrix} \left(\begin{array}{c}
\widetilde{w}_1\cr \hline\hline\hline
\widetilde{w}_2
\end{array}\right) = \left(\begin{array}{c}
Q_M\widetilde{w}_1\cr \hline\hline\hline
\widetilde{w}_2
\end{array}\right).
\end{equation}
The two candidates for the left eigenvector $y$ of quadratic problem are $Q_M\widetilde{w}_1$ and $\widetilde{w}_2$. We choose the one with smaller backward error.
\paragraph{Case 2.1}
The left eigenvectors for zero eigenvalues are represented by the last $n\!-\! r_K$ columns of the matrix $Q_K$. Let $\left(\begin{array}{c}
\widetilde{w}_1\cr \hline\hline\hline
\widetilde{w}_2
\end{array}\right)\!\!\in\!\! \mathbb{C}^{n+r_K}$ be the eigenvector for the deflated pencil (\ref{eq:LGEVP:n+rK}). The corresponding eigenvector for the $2n \times 2n$ pencil, before truncation satisfies
\begin{align*}
&\left(\begin{array}{c|||c c}
\widetilde{w}^*_1 & \widetilde{w}^*_2 & \widetilde{w}^*_3	
\end{array}\right) \left( \left( \begin{array}{c|||c}  X_{11} &  \begin{array}{c|c} X_{12} & X_{13} \end{array} \cr\hline\hline\hline \begin{array}{c} X_{21} \cr\hline \0_{n-r_K,n} \end{array} & 
\begin{array}{c|c} \0_{r_K,r_K} & \0_{r_K,n-r_K}\cr\hline \0_{n-r_K,r_K} & \0 \end{array}\end{array} \right) 
-\lambda \left(\begin{array}{c|||c} -R_M & \0 \cr\hline\hline\hline \0 & \begin{array}{c|c} -\Id_{r_K} & \0 \cr\hline \0 &-\Id_{n-r_K}\end{array}\end{array}\right) \right) = \\
& \left(\begin{array}{c|c|c}
\widetilde{w}^*_1X_{11} + \widetilde{w}^*_2X_{21} + \lambda \widetilde{w}^*_1R_M & 
\widetilde{w}^*_1X_{12} + \lambda \widetilde{w}^*_2 & 
\widetilde{w}^*_1X_{13} + \lambda \widetilde{w}^*_3
\end{array}\right) = \0,\;\mbox{therefore, $\widetilde{w}_3 = X^*_{13}\widetilde{w}_1 / \lambda$.}
\end{align*}
The vector $z$ for $C_2(\lambda)$ is
${\displaystyle \left(\begin{array}{c}
w_1\cr\hline\hline\hline
w_2
\end{array}\right) = \begin{pmatrix}
Q_M & \0\\
\0 & Q_K
\end{pmatrix} \left(\begin{array}{c}
\widetilde{w}_1\cr\hline\hline\hline
\widetilde{w}_2\cr
\widetilde{w}_3
\end{array}\right) = \left(\begin{array}{c}
Q_M\widetilde{w}_1\cr\hline\hline\hline
Q_K\begin{pmatrix}
\widetilde{w}_2\\
\widetilde{w}_3
\end{pmatrix}
\end{array}\right)}$.
The left eigenvector $y$ for the original problem is now chosen between $Q_M\widetilde{w}_1$ and $Q_K\begin{pmatrix}
\widetilde{w}_2\\
\widetilde{w}_3
\end{pmatrix}$. The one with smaller backward error is preferred.
\paragraph{Case 3.1}
The left eigenvectors for zero eigenvalues are represented by the last $n-r_K$ columns of $Q_K$, and for infinite eigenvalues by the last $n-r_M$ columns of $Q_M$. Let $\begin{pmatrix}
\widetilde{w}_1\\
\widetilde{w}_2
\end{pmatrix}\in \mathbb{C}^{r_K+r_M}$ be a left eigenvector for the truncated $(r_K+r_M) \times (r_K+r_M)$ pencil (\ref{eq:deflated:A-lB}). The corresponding eigenvector for the pencil $QC_2(\lambda)V$ satisfies
\begin{align*}
&\left(\begin{array}{c c || c c}
\widetilde{w}^*_1 & \widetilde{w}^*_2 & \widetilde{w}^*_3 & \widetilde{w}^*_4\
\end{array}\right) \left( \left( \begin{array}{c||c}  \begin{array}{c|c} \widetilde{X}_{11} & \widehat{X}_{12}\cr\hline \widetilde{X}_{21} & \widehat{X}_{22}\end{array} &  \begin{array}{c|c} \widehat{X}_{13} & X_{14}\cr\hline \widehat{X}_{23} & \0\end{array}  \cr\hline\hline \begin{array}{c|c} \0 & \0  \cr\hline \0 & \0 \end{array} & 
\begin{array}{c|c} R_X & \widetilde{X}_{24}\cr\hline \0 & \0 \end{array}\end{array} \right)
-\lambda 
\left( \begin{array}{c||c}  
\begin{array}{c|c} \widetilde{Y}_{11} & \widehat{Y}_{12}\cr\hline \widetilde{Y}_{21}  & \widehat{Y}_{22} \end{array} &  \begin{array}{c|c} \widehat{Y}_{13} & \0 \cr\hline \widehat{Y}_{23} & \0\end{array}  \cr\hline\hline \begin{array}{c|c} \0 & \0  \cr\hline \0 & \0 \end{array} & 
\begin{array}{c|c} \0 & \0 \cr\hline \0 & -\Id \end{array}\end{array} \right) \right) = \\
&= \left(\begin{array}{c|c|c|c}
\0 & \0& 
\widetilde{w}^*_1\widehat{X}_{13} + \widetilde{w}^*_2\widehat{X}_{23} + \widetilde{w}^*_3R_X-\lambda\widetilde{w}^*_1 \widehat{Y}_{13} - \lambda \widetilde{w}^*_2\widehat{Y}_{23} & 
\widetilde{w}^*_1 \widehat{X}_{14} + \widetilde{w}^*_3 \widetilde{X}_{24} + \lambda \widetilde{w}^*_4
\end{array} \right) = \0.
\end{align*}
The components $\widetilde{w}^*_3,\widetilde{w}^*_4$ are thus computed as
$$
\widetilde{w}^*_3 = \left(\lambda\widetilde{w}^*_1 \widehat{Y}_{13} + \lambda \widetilde{w}^*_2\widehat{Y}_{23} -\widetilde{w}^*_1\widehat{X}_{13} - \widetilde{w}^*_2\widehat{X}_{23}\right)R^{-1}_X,\;\;
\widetilde{w}^*_4 = \left( -\widetilde{w}^*_1 \widehat{X}_{14} - \widetilde{w}^*_3 \widetilde{X}_{24} \right)/\lambda.
$$
\begin{equation*}
\mbox{The eigenvector for $C_2(\lambda)$ is}\;
w \!=\!\! \begin{pmatrix}
Q_M & \0\\
\0 & Q_K
\end{pmatrix} \!\!\! \left(\begin{smallmatrix}
\mathbb{I}_{r_M} & \0 & \0 & \0\\
\0 & \0 & Q_X & \0 \\
\0 & \mathbb{I}_{r_K} & \0 & \0 \\
\0 & \0 & \0 & \mathbb{I}_{n-r_K}
\end{smallmatrix}\right) \!\!\!\left( \begin{smallmatrix}{c}
\widetilde{w}_1\cr
\widetilde{w}_2\cr \hline\hline
\widetilde{w}_3\cr
\widetilde{w}_4
\end{smallmatrix} \right)\!\! =\!\! \left( \begin{array}{c}
Q_M \left(\begin{smallmatrix}
\widetilde{w}_1\\
Q_X \widetilde{w}_3
\end{smallmatrix}\right)\cr \hline\hline\hline
Q_K \left(\begin{smallmatrix}
\widetilde{w}_2\\
\widetilde{w}_4
\end{smallmatrix}\right)
\end{array} \right),
\end{equation*}
and the candidates for the left eigenvector $y$ are $Q_M \left(\begin{smallmatrix}
\widetilde{w}_1\\
Q_X \widetilde{w}_3
\end{smallmatrix}\right)$ and $Q_K \left(\begin{smallmatrix}
\widetilde{w}_2\\
\widetilde{w}_4
\end{smallmatrix}\right)$. Again, we choose the one with smaller backward error.
\paragraph{Cases 2.2, 3.2 and 3.3} As for the right eigenvectors, in these cases we have the same procedure for the recovery of the left eigenvectors. Write the transformed linearization pencil as
$P(A-\lambda B)Q = \begin{pmatrix}
A_{\ell+j,\ell+j}-\lambda B_{\ell+j,\ell+j} & X\\
\0 & Y
\end{pmatrix}, \;\; {j=1,2,\ell_1}$.
Now, the left eigenvector $w\in \mathbb{R}^{2n}$ for the transformed pencil $P(A-\lambda B)Q$ is
$w = \left(\begin{smallmatrix}
w_1\\
w_2
\end{smallmatrix}\right)$, $w_2 = -w_1^*XY^{-1}$,
where $w_1$ is the computed left eigenvector of $A_{\ell+j,\ell+j}-\lambda B_{\ell+j,\ell+j}$.

For the left eigenvector of the original quadratic problem we choose between $P^*w(1:n)$ and $P^*w(n+1:2n)$, whichever has lower backward error.

Finally, note that we can choose to measure the backward error norm-wise or as more structured as discussed before.
\section{Numerical examples}\label{S=NUMEX}
The potential benefits of the proposed modifications of the \texttt{quadeig} framework have been already illustrated in Figure \ref{fig:EX-mob-manip-2} in Example \ref{EX-Intro-NLEVP-MM}, and in Figure \ref{fig:beam-cbe-right} in \S \ref{SS=B}. 
We now present a summary of testing the proposed method on the NLEVP collection of benchmark examples. With respect to the norm-wise backward error (\ref{eq:back-err-0}), the  comparison of \texttt{quadeig} and \texttt{KVADeig} shows no substantial difference, as expected  -- both use parameter scaling as a default. As noted in \cite{Hammarling:QUADEIG}, \texttt{polyeig} is weaker than \texttt{quadeig} in this respect. Furthermore, we noted that, as shown on the left panel in Figure \ref{fig:beam-cbe-right}, \texttt{quadeig} performs better than \texttt{polyeig} also in terms of the structured backward error. A similar conclusion is obtained with the \texttt{power\_plant} example, see Figure \ref{fig:plant-cbe-right}.

\begin{figure}[H]
	\centering
	\includegraphics[width=0.45\textwidth, height=1.6in]{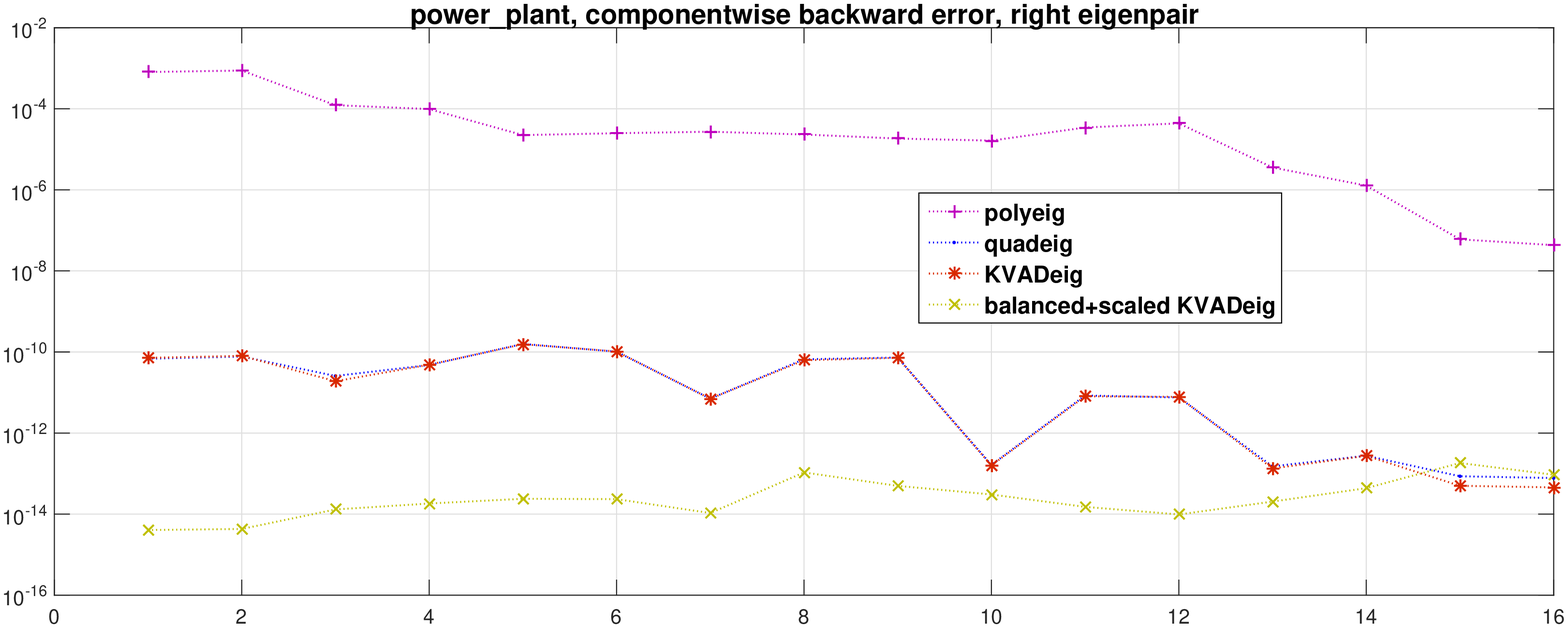}
	\includegraphics[width=0.45\textwidth, height=1.6in]{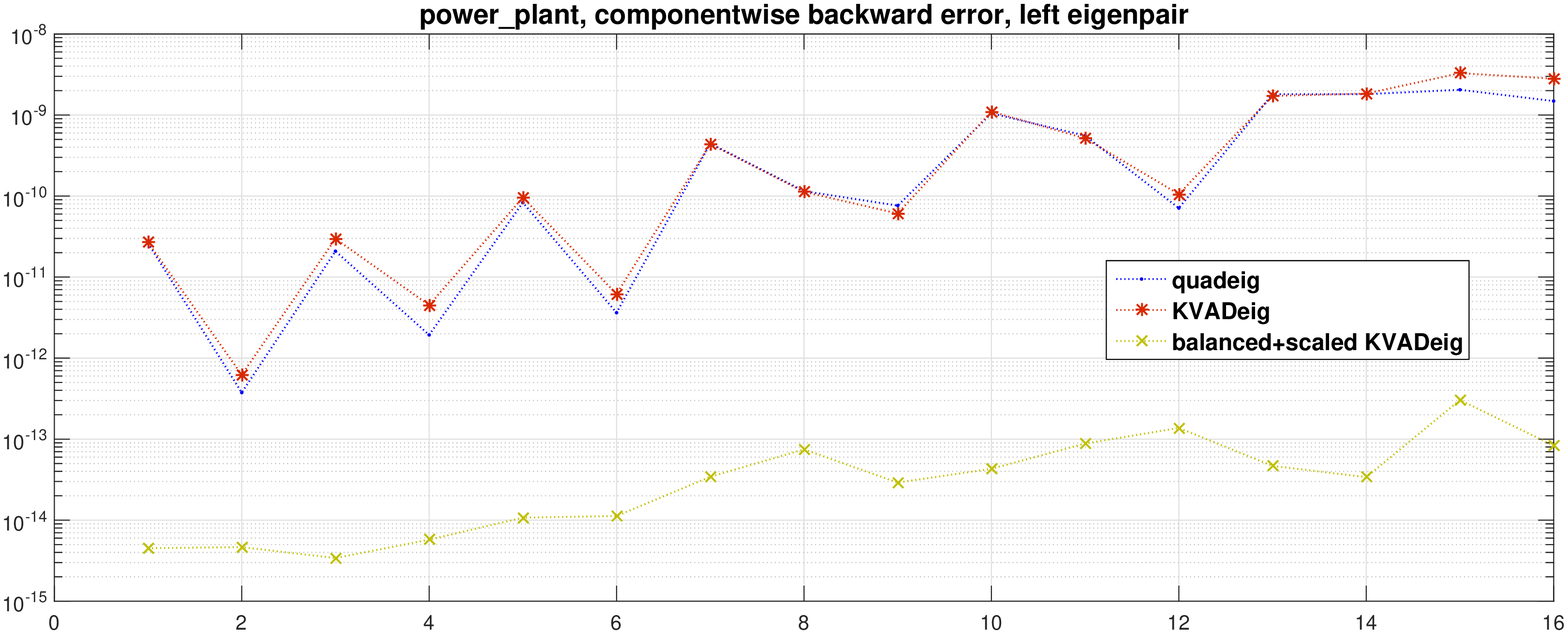}
	\caption{\label{fig:plant-cbe-right} The example \texttt{power\_plant} from the NLEVP collection. Note how both scaling and balancing contribute to reduction of the component-wise backward error.}	
\end{figure}	
\section{Concluding remarks}\label{S=Conclusion}
We have proposed a modification of the \texttt{quadeig} algorithm \cite{Hammarling:QUADEIG}. In addition to the parameter scaling, we use diagonal balancing and allow the rank revealing QR factorizations with full pivoting; this is shown to enhance numerical/backward stability as measured by structured backward error. Our backward error analysis is instructive -- it shows the mechanism of keeping the backward error small in the coefficient matrices of the original problem in the first two steps of the deflation process (thus including \texttt{quadeig}), and it also shows how the proposed modifications contribute to keeping the structured backward error small. Further we extend the deflation procedure of \texttt{quadeig} toward an upper triangular variation of the Kronecker Canonical Form, which opens the possibility of removing additional infinite eigenvalues that were beyond the reach of \texttt{quadeig}. We believe that our contributions to the reduction framework introduced by \texttt{quadeig} will be useful for both the full solution of medium size problems and the iterative methods (e.g. SOAR) where full solution of the projected problem is needed and where notrivial techniques of deflation, purging and locking may be critically dependent on it. These issues are the subject of our ongoing and future work. 
A LAPACK--style implementation of the method, with further fine details of numerical software development is in progress.
\vspace{-2mm}
\section*{Acknowledgments}
This research was part of the second authors thesis \cite{ISG-thesis-2018} and of the project IP-11-2013-9345 (\emph{Mathematical Modelling, Analysis and Computing With Applications to Complex Mechanical Systems}) supported by the Croatian Science Foundation.
The authors thank Serkan Gugercin (Virginia Tech), Luka Grubi\v{s}i\'{c} and Zvonimir Bujanovi\'{c} (University of Zagreb) for valuable comments. 

\bibliography{KVADEIG_References}
\bibliographystyle{plain} 

\end{document}